\documentclass[a4paper]{article}

\usepackage{adjustbox, amsfonts, amsmath, amssymb, amsthm, authblk, color, enumerate, enumitem, graphicx, mathrsfs, mathtools, tikz-cd}

\usetikzlibrary{calc, cd, decorations.markings, decorations.pathmorphing, decorations.pathreplacing, fit, positioning, shapes, shapes.geometric, tikzmark}

\usepackage{rotating, lscape}

\usepackage{color}
\definecolor{myurlcolor}{rgb}{0.1,0.1,0.8}
\definecolor{mylinkcolor}{rgb}{0.05,0.05,0.4}

\usepackage{hyperref}
\hypersetup{colorlinks,
	linkcolor=mylinkcolor,
	citecolor=mylinkcolor,
	urlcolor=myurlcolor}

\usepackage{cleveref}

\usepackage{tocloft}
\newtheoremstyle{dotless_thm}{}{}{\itshape}{}{\bfseries}{}{ }{}
\newtheoremstyle{dotless_def}{}{}{}{}{\bfseries}{}{ }{}
\newtheoremstyle{dotless_blue}{}{}{\color{blue}}{}{\bfseries}{}{ }{}
\newtheoremstyle{dotless_red}{}{}{\color{red}}{}{\bfseries}{}{ }{}

\theoremstyle{dotless_thm}
\newtheorem{thm}{Theorem}
\newtheorem*{thm*}{Theorem}
\numberwithin{thm}{section}
\newtheorem{cor}[thm]{Corollary}
\newtheorem*{cor*}{Corollary}
\newtheorem{lem}[thm]{Lemma}
\newtheorem*{lem*}{Lemma}
\newtheorem{prop}[thm]{Proposition}
\newtheorem*{prop*}{Proposition}

\newtheorem*{thmgt_mag}{Theorem~\ref{thm:gt_mag}} 
\newtheorem*{propex_acyc}{Proposition~\ref{prop:ex_acyc}} 

\theoremstyle{dotless_def}
\newtheorem{defn}[thm]{Definition}

\newtheorem{example}[thm]{Example}

\newtheorem{rmk}[thm]{Remark}

\definecolor{violet}{HTML}{990099}
\definecolor{indigo}{HTML}{6600cc}
\definecolor{paleindigo}{HTML}{eacffa}
\definecolor{blue}{HTML}{0000cc}
\definecolor{paleblue}{HTML}{bbb6fa}
\definecolor{green}{HTML}{009900}
\definecolor{yellow}{HTML}{ffff00}
\definecolor{orange}{HTML}{ff3300}
\definecolor{red}{HTML}{ff0000}
\definecolor{palegreen}{HTML}{e6fff7}
\definecolor{darkgreen}{HTML}{006633}

\newcommand{\bb}{\mathbb}
\newcommand{\bdiag}{\begin{equation*}\begin{tikzcd}}
\newcommand{\ediag}{\end{tikzcd}\end{equation*}}

\newcommand{\Mag}{\mathrm{Mag}}
\renewcommand{\th}{\textrm{th}}
\newcommand{\rk}{\mathrm{rk}}
\renewcommand{\vec}{\mathbf}

\tikzset{
    invisible/.style={opacity=0},
    visible on/.style={alt={#1{}{invisible}}},
    alt/.code args={<#1>#2#3}{%
      \alt<#1>{\pgfkeysalso{#2}}{\pgfkeysalso{#3}}%
  }
}

\usepackage[style=trad-abbrv]{biblatex}
\title{Magnitude, homology, and the Whitney twist}

\author{Emily Roff}

\date{}

\begin{document}

\maketitle

\begin{abstract}
Magnitude is a numerical invariant of metric spaces and graphs, analogous, in a precise sense, to Euler characteristic. Magnitude homology is an algebraic invariant constructed to categorify magnitude. Among the important features of the magnitude of graphs is its behaviour with respect to an operation known as the Whitney twist. We give a homological account of magnitude's invariance under Whitney twists, extending the previously known result to encompass a substantially wider class of gluings. As well as providing a new tool for the computation of magnitudes, this is the first new theorem about magnitude to be proved using magnitude homology.
\end{abstract}

\tableofcontents


\section{Introduction}\label{sec:intro}

\emph{Magnitude} is an isometric invariant of metric spaces, so-named for its web of connections to `size-like' quantities of significance in various corners of mathematics \cite{LeinsterMagnitude2013}. The magnitude of a space \(X\) is a function 
\[\Mag(X)(-): [0, \infty) \to \bb{R}\]
whose parameter can be thought of as controlling the scale of the metric in \(X\). For any given choice of the parameter, magnitude behaves in many ways like the cardinality of finite sets: it is multiplicative with respect to \(\ell_1\)-products, additive with respect to disjoint unions, and satisfies an inclusion-exclusion formula with respect to a more general class of unions. Yet, unlike cardinality, magnitude is sensitive to the distances in a space.

\begin{figure}
\centering
\includegraphics[width=0.95\textwidth]{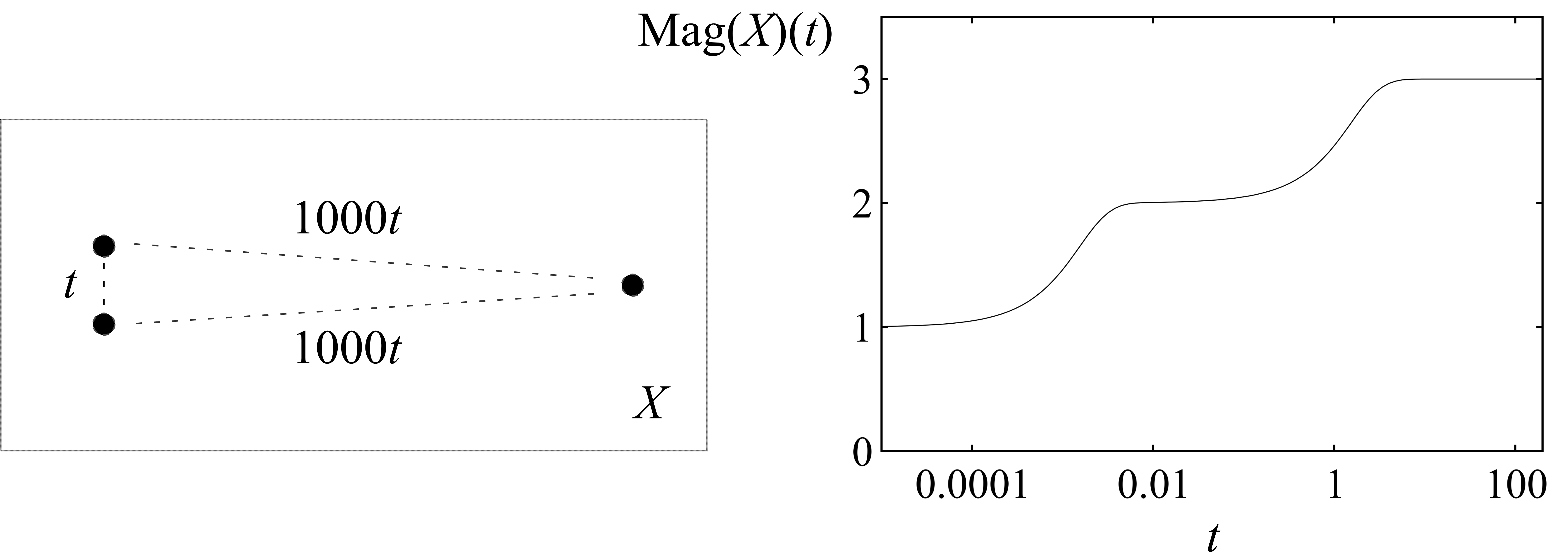}
\caption{The magnitude function of a three-point space. (Example 6.4.6 in Leinster \cite{LeinsterEntropy2021}, originally due to Willerton.)}
\label{fig:3pt}
\end{figure}

The three-point space \(X\) in \Cref{fig:3pt} is indicative of the general situation for finite spaces. When the scale parameter \(t\) is very small, the magnitude function is close to 1, as if recognizing that `from far away' \(X\) resembles a one-point space. As \(t\) increases, it becomes possible to distinguish the point on the right, and \(\Mag(X)(t)\) lingers close to 2. Once \(t\) is sufficiently large, all three points in the space are easily distinguished and \(\Mag(X)(t)\) approaches 3. Accordingly, the magnitude function is often referred to as counting `the effective number of points' in a finite space as the scale of the metric is varied.

Magnitude can be defined for many non-finite compact metric spaces, too, and in this context it has been shown to be related to a remarkable variety of more classical size-related quantities. Under suitable conditions, the magnitude function records information about the intrinsic volumes of integral geometry \cite{BarceloCarbery2018, GimperleinGoffeng, GimperleinGoffengLouca2022, LeinsterMagnitude2013b, MeckesMagnitude2019}, curvature and the Willmore energy \cite{GimperleinGoffengWillmore}, Minkowski dimension \cite{MeckesMagnitude2015} and even a family of indices used by biologists to quantify the diversity of ecological communities \cite{LeinsterMeasuring2012, LeinsterMaximizing2016, LeinsterRoffMaximum2021}. (These references are far from exhaustive.)

In this paper, though, we will be concerned specifically with the magnitude of graphs. Every undirected finite graph can be regarded as a metric space in which the distance between any pair of vertices is the number of edges in a shortest edge-path joining them. The magnitude of a graph is defined to be its magnitude with respect to this metric.

In this context, magnitude is most naturally interpreted as a formal power series with integer coefficients. Leinster  \cite{LeinsterMagnitude2019} gives a combinatorial formula for the coefficients in this series, revealing them to count---with appropriate signs---the paths of given length in a graph. (This terminology will be explained in \Cref{sec:magnitude}.) Observing that the alternating sum computing each coefficient resembles the Euler characteristic of a chain complex, Hepworth and Willerton in \cite{HepworthWillerton2017} construct a bigraded homology theory for graphs which `categorifies' their magnitude. That theory, \emph{magnitude homology}, has since been extended by Leinster and Shulman \cite{LeinsterMagnitude2017v3} to encompass all metric spaces. In what follows we denote the magnitude homology functor by \(MH_\bullet^*\), with \(\bullet\) and \(*\) standing for the two distinct gradings; its construction is described in \Cref{sec:magnitude}.

The magnitude homology of metric spaces has been well studied in recent years, and has been shown to capture subtle information about the convexity of a space \cite{KanetaYoshinaga2018} and the existence and uniqueness of geodesics \cite{AsaoMH2019, GomiGeodesic2019}. Even in the restricted setting of graphs it has proved to be a rich invariant. Particular attention has been paid to a class of graphs which are termed \emph{diagonal}: graphs \(X\) with the property that \(MH_k^\ell(X) = 0\) whenever \(k \neq \ell\) \cite{HepworthWillerton2017, GuGraph2018}. Most recently, Asao \cite{AsaoMH2022} has established a close relationship between magnitude homology and Grigor'yan--Muranov--Yau \emph{path homology}, exploiting this relationship to prove that every diagonal graph has trivial reduced path homology.

Meanwhile, various results about the formal cardinality-like properties of magnitude have been recovered as the numerical `shadows' of facts about magnitude homology. In particular, Bottinelli and Kaiser \cite{BottinelliKaiser2021} (following Hepworth and Willerton \cite{HepworthWillerton2017}) have proved a K\"unneth-type theorem and a Mayer--Vietoris theorem for the magnitude homology of metric spaces---categorifying, respectively, the mutiplicativity of magnitude and the inclusion-exclusion formula. This paper takes a similar form, offering a homological account of magnitude's behaviour with respect to an operation on graphs known as the {Whitney twist}.

\paragraph{Magnitude and the Whitney twist}

In Leinster's paper \cite{LeinsterMagnitude2019} initiating the study of magnitude for graphs, there are three central results concerning magnitude's behaviour under combinations such as products or gluings. The first of these, Lemma 3.6, says that magnitude is multiplicative with respect to the operation that graph theorists call the \emph{cartesian product}, \(\square\). That is, for any graphs \(X\) and \(Y\),
\begin{equation}\label{eq:mag_mult}
\Mag(X \square Y) = \Mag(X) \cdot \Mag(Y).
\end{equation}
The second such result, Theorem 4.9 in \cite{LeinsterMagnitude2019}, gives an inclusion-exclusion formula for the magnitude of a union of graphs, subject to conditions on the subgraphs involved. Each of these results is a specialization to graphs of a corresponding property of metric space magnitude.

The third combination result in \cite{LeinsterMagnitude2019} is different. It concerns an operation known as the \emph{Whitney twist}, which can be described as follows. Given two graphs \(G\) and \(H\), each with two (distinct) distinguished vertices, there are two ways to form a new graph: we can glue the distinguished vertices together one way round, or the other. If \(X\) and \(Y\) are the graphs thus formed, one says that \(X\) and \(Y\) differ by a Whitney twist. Leinster proves that magnitude is preserved under Whitney twists---\emph{provided} the gluing vertices happen to be adjacent.
\begin{thm}[\cite{LeinsterMagnitude2019}, Theorem 5.2]\label{thm:whitney_mag}
Let \(X\) and \(Y\) be graphs differing by a Whitney twist, and suppose that the two gluing vertices are adjacent in \(X\) (or equivalently in \(Y\)). Then
\[\Mag(X) = \Mag(Y).\]
\end{thm}
It can be seen from Example 4.16 in the same paper that magnitude need \emph{not} be preserved under Whitney twists in which the gluing vertices are not adjacent. From this it follows that the magnitude of a graph---unlike many other graph invariants---is not determined by the cycle matroid or the Tutte polynomial. 

While multiplicativity and the inclusion-exclusion formula both lend weight to the interpretation of magnitude as a `size-like' invariant, the fact that magnitude {is} preserved under certain Whitney twists but not others is intriguing, and thus far poorly understood. In particular, it is not easy to extract from the proof of Leinster's Theorem 5.2 a clear understanding of what exactly it is about the adjacency of the two gluing vertices that guarantees the result, leading one to wonder whether the theorem might be merely an instance of a more general statement about  gluings. In this paper we prove that, in fact, it is.

Moreover, whereas the multiplicativity of magnitude and the inclusion-exclusion formula have been categorified, respectively, by the K\"unneth theorem and the Mayer--Vietoris sequence for magnitude homology, magnitude's behaviour under Whitney twists has until now resisted a homological explanation. Indeed, the behaviour of magnitude {homology} under Whitney twists is listed as an open question in \cite{HepworthWillerton2017}.

That question is not settled here; instead we prove, by homological means, an extension of \Cref{thm:whitney_mag} to a wider class of `twisted' graph gluings which we term \emph{sycamore twists} (\Cref{def:syc_twist}). Our main theorem runs as follows.

\begin{thmgt_mag}
Let \(X\) and \(Y\) be graphs which differ by a sycamore twist. Then
\[\Mag(X) = \Mag(Y).\]
\end{thmgt_mag}

As well as providing a new tool for computing magnitudes of gluings, \Cref{thm:gt_mag} can be regarded as clarifying the picture of graph magnitude in two senses---and in doing so it opens up new questions about magnitude homology.

First, the definition of a sycamore twist distills the important features of a Whitney twist with adjacent gluing vertices and elucidates the role played by this adjacency. Second, the proof of the theorem makes it clear that the behaviour of magnitude under twists is intimately related to the excision theorem and the Mayer--Vietoris sequence for magnitude homology and thereby to the inclusion-exclusion formula for magnitude.

Indeed, the result opens up the possibility that the excision and Mayer--Vietoris theorems might yet be strengthened, implying a more general inclusion-exclusion principle for magnitude. This would afford efficiency in computations, but it could also prove valuable for theoretical reasons. As things stand, the excision theorem is not strong enough to provide an axiomatization of magnitude homology for any class of graphs or metric spaces large enough to be interesting. A stronger excision theorem might allow for an axiomatic theory of magnitude homology and thereby of magnitude for finite spaces.

Finally, while various authors have obtained categorifications of known results about magnitude---and there are many interesting results about magnitude homology which do not appear to have immediate consequences for magnitude---to the best of our knowledge \Cref{thm:gt_mag} is the first \emph{new} result about magnitude to be proved using magnitude homology. We believe this offers reason to hope that other mysteries in the theory of magnitude may eventually be resolvable using homological techniques.

\paragraph{The structure of this paper}

We begin, in Section \ref{sec:magnitude}, with an overview of magnitude and magnitude homology. \Cref{sec:good_twist} introduces the main object of interest: a pair of graphs related by a sycamore twist. Sections \ref{sec:props_good_twist}--\ref{sec:excise} are dedicated to the proof of the main theorem.

\paragraph{Acknowledgements}

I am grateful to Tom Leinster for discussions throughout the development of this work, and to Richard Hepworth for his comments on the early draft of this paper which forms part of my doctoral thesis. Thanks are due also to Sebastian Schlegel Mejia for helpful conversations concerning homological algebra, and to the anonymous referee, whose attentive reading and suggestions have improved the exposition in several places.


\section{Magnitude and magnitude homology}\label{sec:magnitude}

Throughout this paper, a \textbf{graph} is an undirected finite graph with no loops or multiple edges. We endow a graph \(X\) with the metric in which the distance \(d(u,v)\) (or \(d_X(u,v)\)) from a vertex \(u\) to a vertex \(v\) is the number of edges in a minimal edge-path connecting them, or \(\infty\) if no such edge-path exists.

Let \(\bb{Z}[q]\) denote the ring of polynomials over \(\bb{Z}\) in one variable, \(q\). Given a graph \(X\), we construct a square matrix \(Z_X\) whose rows and columns are indexed by the vertices of \(X\) and whose \((u,v)\) entry is 
\[Z_X(u,v) = q^{d(u,v)} \in \bb{Z}[q],\]
adopting the convention that \(q^\infty = 0\).

The diagonal entries of the matrix \(Z_X\) are all equal to \(1\), and each of its off-diagonal entries is either zero or \(q^n\) for some natural number \(n \neq 0\). The determinant of \(Z_X\) is thus a polynomial in \(q\) with constant term 1, so is a unit in the ring \(\bb{Z}[[q]]\) of power series in \(q\) with integer coefficients. This ensures that \(Z_X\) is invertible over \(\bb{Z}[[q]]\), allowing us to make the following definition.

\begin{defn}[Leinster \cite{LeinsterMagnitude2019}, Definition 2.1]\label{defn:magnitude}
The \textbf{magnitude} of a graph \(X\) is
\[\Mag(X) = \sum_{u,v \in X} Z_X^{-1}(u,v) \in \bb{Z}[[q]].\]
\end{defn}

The magnitude of a graph is a specialization of an invariant defined in the vastly greater generality of finite enriched categories (\cite{LeinsterMagnitude2013}, Section 1). Specialized in a different direction, to ordinary categories, magnitude turns out to have close links to topological Euler characteristic (\cite{LeinsterMagnitude2008}, Propositions 2.11 and 2.12); interpreted for posets, it extends the theory of M\"obius inversion. Perhaps its most fertile environment, though, is that of metric spaces regarded as categories enriched in the poset \(([0, \infty], \geq)\) with monoidal structure given by addition. (For the classical account of this perspective on metric spaces, see Lawvere \cite{LawvereMetric1974}.) An overview of the properties of metric space magnitude can be found in \cite{LeinsterMagnitude2017b}.

Evaluating \(\Mag(X)\) at \(q=e^{-t}\) for all nonnegative real numbers \(t\) yields the \textbf{magnitude function} discussed in the introduction: 
\[\Mag(X)(t) = \Mag(X)|_{q = e^{-t}}.\]
In general this function may diverge at some \(t\); Example 2.2.7 in \cite{LeinsterMagnitude2013} describes a graph exhibiting this behaviour. In what follows, we will deal exclusively with the power series \(\Mag(X)\).

Given magnitude's relationship to Euler characteristic, it is natural to ask whether there exists an algebraic invariant which `categorifies' magnitude in the sense that singular homology categorifies the Euler characteristic of a topological space. In 2015 Hepworth and Willerton answered this question affirmatively in the case of graphs, constructing a homology theory whose Euler characteristic recovers their magnitude \cite{HepworthWillerton2017}; in 2017 Leinster and Shulman extended that construction to a broad class of enriched categories, including all metric spaces \cite{LeinsterMagnitude2017v3}. In this paper we will be working exclusively with Hepworth and Willerton's \emph{magnitude homology} for graphs, which we now describe.

Their construction depends on a combinatorial formula for the coefficients in the power series \(\Mag(X)\), given by Leinster in \cite{LeinsterMagnitude2019}. To state the formula, it will be helpful to make the following definition.

\begin{defn}\label{def:paths}
A \textbf{path}, or a \textbf{\(k\)-path}, in a graph \(X\) is a tuple \(\vec{x} = (x_0, \ldots, x_k)\) of vertices in \(X\). We call \(\vec{x}\) \textbf{non-degenerate} if \(x_0 \neq x_1 \neq \cdots \neq x_k\). The \textbf{length} of a path \(\vec{x}\) is
\[L(\vec{x}) = \sum_{i=0}^{k-1} d(x_i, x_{i+1}).\]
\end{defn}

\begin{rmk}\label{rmk:count}
Notice we do {not} require that consecutive vertices in a path be connected by an edge; thus, every non-degenerate \(k\)-path in a graph \(X\) has length at least \(k\), and possibly greater. Since all our graphs are finite, there can be only finitely many non-degenerate paths in \(X\) of any given length; moreover, there can be \emph{no} non-degenerate \(k\)-paths of length \(\ell\) for \(k > \ell\).
\end{rmk}

Leinster's formula says that calculating the magnitude of a graph \(X\) comes down to counting non-degenerate paths in \(X\) of every possible length. 

\begin{prop}[\cite{LeinsterMagnitude2019}, Proposition 3.9]\label{thm:mag_coeffs}
For any graph \(X\),
\begin{equation}\label{eq:mag_series}
\Mag(X) = \sum_{\ell=0}^\infty c_\ell(X) q^\ell
\end{equation}
where the coefficients are given by
\begin{equation}\label{eq:mag_coeffs}
c_\ell(X) = \sum_{k=0}^\ell (-1)^k \#\{\text{non-degenerate \(k\)-paths of length \(\ell\) in X}\}.
\end{equation}
\end{prop}

Hepworth and Willerton proceed from \Cref{thm:mag_coeffs}, constructing an \(\bb{N}\)-graded chain complex whose Euler characteristic in grading \(\ell \in \bb{N}\) computes the coefficient \(c_\ell\) in Leinster's formula.

\begin{defn}[\cite{HepworthWillerton2017}, Definition 2]\label{def:MC_graphs}
The \textbf{magnitude chain complex} of a graph \(X\) is the direct sum of chain complexes
\[MC(X) = \bigoplus_{\ell \in \bb{N}} MC_\bullet^\ell(X)\]
where the chain complex \(MC_\bullet^\ell(X)\) is freely generated in degree \(k \geq 0\) by the set of non-degenerate \(k\)-paths in \(X\) whose length is \(\ell\). The boundary operator
\[d: MC_k^\ell(X) \to MC_{k-1}^\ell(X)\]
is an alternating sum \(d = \sum_{i=1}^{k-1} (-1)^i \delta_i\) where \(\delta_i\) is defined on generators by
\[\delta_i(x_0, \ldots, x_k) = \begin{cases} (x_0, \ldots, \widehat{x_i}, \ldots, x_k) & \text{if } L(x_0, \ldots, \widehat{x_i}, \ldots, x_k) = \ell \\
0 & \text{otherwise}.\end{cases}.\]
Here, \((x_0, \ldots, \widehat{x_i}, \ldots, x_k)\) denotes the path \((x_0, \ldots, x_{i-1}, x_{i+1}, \ldots, x_k)\).
\end{defn}

\begin{defn}[\cite{HepworthWillerton2017}, Definition 3]\label{def:MH_graphs}
The \textbf{magnitude homology} of a graph \(X\) is the bigraded abelian group defined for \((k,\ell) \in \bb{N} \times \bb{N}\) by
\[MH_k^\ell(X) = H_k(MC_\bullet^\ell(X)).\]
We will refer to \(\ell\) as the \textbf{length grading} and \(k\) as the \textbf{homological degree} of the abelian groups \(MC_k^\ell(X)\) and \(MH_k^\ell(X)\).
\end{defn}

That the magnitude of a graph can be recovered from its magnitude chain complex follows immediately from \Cref{thm:mag_coeffs}.

\begin{thm}[\cite{HepworthWillerton2017}, Theorem 2.8]\label{thm:MH_mag}
Let \(X\) be a graph. Then
\[\Mag(X) = \sum_{\ell \in \bb{N}} \chi(MC_\bullet^\ell(X)) q^\ell.\]
\end{thm}

The construction of the magnitude chain complex---and thus magnitude homology---is functorial with respect to morphisms of graphs which preserve or contract edges; equivalently, vertex functions \(f: X \to Y\) satisfying
\[d_Y(f(x), f(x')) \leq d_X(x,x')\]
for all pairs of vertices \(x,x'\) in \(X\). These \textbf{distance-decreasing} maps are especially natural to consider when regarding a metric space as an enriched category.

The chain map induced by a distance-decreasing function \(f: X \to Y\) is given on each generator \(\vec{x}=(x_0,\ldots,x_k)\) of \(MC(X)\) by
\[MC(f)(\vec{x}) = \begin{cases}
f(\vec{x}) & \text{if } L(f(\vec{x})) = L(\vec{x}) \\
0 & \text{otherwise},
\end{cases}\]
where \(f(\vec{x})\) denotes the path \((f(x_0),\ldots,f(x_k))\) in \(Y\). That is,
\[MC(f)(\vec{x}) = \begin{cases}
f(\vec{x}) & \text{if } d_Y(f(x_i),f(x_{i+1})) = d_X(x_i,x_{i+1}) \text{ for } 0 \leq i < k \\
0 & \text{otherwise}.
\end{cases}\]
In other words, the induced map \(MC(f): MC(X) \to MC(Y)\) retains information only about those paths in \(X\) which are mapped into \(Y\) by \(f\) in a strictly {length-preserving} manner.

On the other hand, lifting magnitude from an element of a ring of formal power series to an object of a category of chain complexes grants us the ability to study it homologically. By \Cref{rmk:count} the chain complex \(MC_\bullet^\ell(X)\) vanishes in homological degrees \(k > \ell\), while in degrees \(k \leq \ell\) it is finitely generated. Standard facts of homological algebra imply, then, that the Euler characteristic of the complex \(MC_\bullet^\ell(X)\) coincides with that of its homology. Thus, we can compute the \(\ell^\th\) coefficient in the power series \(\Mag(X)\) as
\[c_\ell = \sum_{k=0}^\ell (-1)^k \rk(MH_k^\ell(X)).\]
That is,
\begin{equation}\label{eq:mag_MH}
\Mag(X) = \sum_{\ell \in \bb{N}} \chi(MH_\bullet^\ell(X)) q^\ell.
\end{equation}
With this formula, homological techniques can be used to simplify the calculation of the coefficients in \(\Mag(X)\). This is the principle we will be applying in later sections of this paper. 

Before introducing the main object of interest, we briefly review the inclusion-exclusion theorem for magnitude and the Mayer--Vietoris theorem for the magnitude homology of graphs. Both theorems hold under the same conditions, captured by the next three definitions.

\begin{defn}[\cite{LeinsterMagnitude2019}, Definition 4.2]\label{def:convex}
A subgraph \(W \subseteq X\) is called \textbf{convex} if \(d_W(u,v) = d_X(u,v)\) for all pairs of vertices \(u,v \in W\). In other words, \(W\) is convex if the inclusion \(W \hookrightarrow X\) is an isometric embedding.
\end{defn}

\begin{defn}\label{def:vertices_project}
Let \(W \subseteq X\) be a convex subgraph. We say that a vertex \(v \in X\) \textbf{projects} to \(W\) if \(v\) is connected by an edge-path to some vertex in \(W\) and there exists a vertex \(\pi(v) \in W\) such that
\[d(v,w) = d(v,\pi(v)) + d(\pi(v),w)\]
for every vertex \(w\) in \(W\). We say a subgraph \(Y \subseteq X\) \textbf{projects} to \(W\) if  \(Y \cap W \neq \emptyset\) and every vertex in \(Y\) which is connected by an edge-path to \(W\) projects to \(W\).
\end{defn}

Note that, if a vertex \(v\) projects to \(W\), then \(\pi(v)\) is the \emph{unique} vertex of \(W\) closest to \(v\). Thus, writing \(U_W\) for the set of vertices in \(X\) which project to \(W\), we have a function \(\pi: U_W \to W\).

\begin{defn}[\cite{HepworthWillerton2017}, Definition 26]\label{def:proj_decomp}
A \textbf{projecting decomposition} is a triple \((X;G,H)\) where \(X\) is a graph with subgraphs \(G\) and \(H\) such that the following properties hold.
\begin{itemize}
\item \(X = G \cup H\).
\item \(G \cap H\) is convex in \(X\).
\item \(H\) projects to \(G \cap H\).
\end{itemize}
\end{defn}

The inclusion-exclusion formula for the magnitude of graphs (Leinster \cite{LeinsterMagnitude2019}, Theorem 4.9) says that if \((X; G,H)\) is a projecting decomposition then
\begin{equation*}
\Mag(X) = \Mag(G) + \Mag(H) - \Mag(G \cap H).
\end{equation*}
The Mayer--Vietoris theorem for magnitude homology (Hepworth and Willerton \cite{HepworthWillerton2017}, Theorem 29) categorifies this formula by a split short exact sequence
\begin{align*}\label{eq:mayer_vietoris}
0 \to MH_\bullet^*(G \cap H) \to MH_\bullet^*(G) \oplus MH_\bullet^*(H) \to MH_\bullet^*(X) \to 0.
\end{align*}

Hepworth and Willerton do {not} attempt to categorify Leinster's theorem concerning Whitney twists (\Cref{thm:whitney_mag} in this paper). Instead, they pose the question: do two graphs related by a Whitney twist along adjacent gluing vertices have isomorphic magnitude homology?

That question will not be answered here. Rather, we will give a homological proof of Leinster's theorem which does not depend on the existence of an isomorphism of magnitude homologies. Our main theorem will extend Leinster's result to encompass a wider class of twisted gluings we term \emph{sycamore twists}. The next section introduces these.


\section{Sycamore twists}\label{sec:good_twist}

We begin by stating the definition of a Whitney twist more formally.

\begin{defn}\label{def:whitneytwist}
Let \(G\) be a graph with two distinct distinguished vertices \(g_+\) and \(g_-\), and \(H\) a graph with distinct distinguished vertices \(h_+\) and \(h_-\). Form a new graph \(X\) by taking the disjoint union of \(G\) and \(H\) and identifying \(g_+\) with \(h_+\) and \(g_-\) with \(h_-\), then identifying any double edges that result. Form another graph \(Y\) in the same manner, this time identifying \(g_+\) with \(h_-\) and \(g_-\) with \(h_+\). The graphs \(X\) and \(Y\) then differ by a \textbf{Whitney twist}.
\end{defn}

\begin{figure}[h]
\adjustbox{scale=0.7, center}{
\begin{tikzcd}
& & \color{blue}{g_0} \arrow[blue, no head]{dl} \arrow[blue, no head]{dr} & & \color{violet}{h_0} \arrow[violet, no head]{dl} \arrow[violet, no head]{dr}\\
& \color{blue}{g_1} \arrow[blue, no head]{rr} \arrow[blue, no head]{ddrr} & & \color{green}{k_0}  \arrow[green, no head]{dd} \arrow[violet, no head]{rr} & & \color{violet}{h_1} \\
\color{blue}{g_2} \arrow[blue, no head]{ur}  \arrow[blue, no head]{dr} &&&&&& \color{violet}{h_2} \arrow[violet, no head]{ul}  \arrow[violet, no head]{dl} \\
& \color{blue}{g_3} \arrow[blue, no head]{rr} & & \color{green}{k_1}  \arrow[violet, no head]{rr}  \arrow[violet, no head]{uurr} & & \color{violet}{h_3} \\
X
\end{tikzcd}
\begin{tikzcd}
& & \color{blue}{g_0} \arrow[blue, no head]{dl} \arrow[blue, no head]{dr}\\
& \color{blue}{g_1} \arrow[blue, no head]{rr} \arrow[blue, no head]{ddrr} & & \color{green}{k_0}  \arrow[green, no head]{dd} \arrow[violet, no head]{rr}  \arrow[violet, no head]{ddrr} & & \color{violet}{h_3} \\
\color{blue}{g_2} \arrow[blue, no head]{ur}  \arrow[blue, no head]{dr} &&&&&& \color{violet}{h_2} \arrow[violet, no head]{ul}  \arrow[violet, no head]{dl} \\
& \color{blue}{g_3} \arrow[blue, no head]{rr} & & \color{green}{k_1}  \arrow[violet, no head]{rr} & & \color{violet}{h_1} \\
 & & & & \color{violet}{h_0} \arrow[violet, no head]{ul} \arrow[violet, no head]{ur} && Y
\end{tikzcd}
}
\caption{A Whitney twist along adjacent gluing vertices. Blue vertices and edges belong to \(G\); violet vertices and edges belong to \(H\); and green vertices and edges belong to \(G \cap H\). Notice that the graphs \(X\) and \(Y\) are non-isometric and that neither \((X;G,H)\) nor \((Y;G,H)\) is a projecting decomposition, since the vertices \(g_1, g_2, h_1\) and \(h_2\) do not project to \(G \cap H\).}\label{fig:whit_1}
\end{figure}
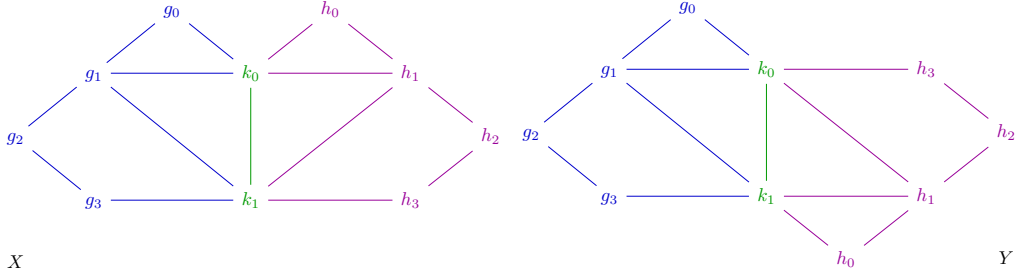

Consider the following straightforward generalization of \Cref{def:whitneytwist}.

\begin{defn}
Let \(G\), \(H\) and \(K\) be graphs equipped with induced subgraph inclusions
\[G \xhookleftarrow{\iota_G} K \xhookrightarrow{\iota_H} H\]
and let \(\alpha:K \to K\) be an isometry. Form a new graph \(X\) by taking the disjoint union \(G \sqcup H\) and identifying the vertices \(\iota_G(v)\) and \(\iota_H(v)\) for each \(v \in K\); wherever a double edge is created, identify the two edges. Form another graph \(Y\) in the same way, but identifying \(\iota_G(v)\) with \(\iota_H(\alpha(v))\) for each \(v \in K\). We will say that \(X\) and \(Y\) differ by a \textbf{generalized Whitney twist}, or just by a \textbf{twist}, and specify the twist by the tuple \((G, H, K, \alpha)\).
\end{defn}

There are two scenarios in which we can already be sure that two graphs differing by a twist will have the same magnitude---indeed, the same magnitude homology. One relates to the Mayer--Vietoris sequence. If \(K\) is convex in \(X\) and \(Y\), and \(H\) projects to \(K\), then \((X; G,H)\) is a {projecting decomposition} and so is \((Y;G,H)\). In that case the Mayer--Vietoris sequence tells us that the homology of both \(X\) and \(Y\) is determined by that of \(G\), \(H\) and \(K\); in particular, we have
\begin{equation}\label{eq:MHX_MHY}
MH_{\bullet}^*(X) \cong MH_{\bullet}^*(Y)
\end{equation}
and thus \(\Mag(X) = \Mag(Y)\).

On the other hand, suppose \(\alpha\) extends to a self-isometry of \(H\) which fixes every vertex in \(H \backslash K\)---that is, suppose 
\begin{equation}\label{eq:equidist}
d_H(v, k) = d_H(v, \alpha(k))
\end{equation}
for all \(v \in H \backslash K\) and all \(k \in K\). In this case \(X\) and \(Y\) are isometric, so certainly their magnitude homology agrees---though in general we cannot expect to compute it from the homologies of \(G\), \(H\) and \(K\).

For graphs differing by a Whitney twist---even one with adjacent gluing vertices---neither of these scenarios necessarily applies (see \Cref{fig:whit_1}). However, a Whitney twist along adjacent gluing vertices does satisfy a hybrid of the two conditions. In fact, Leinster's proof of the invariance of magnitude under such twists relies on the observation that the vertex sets of \(X\) and \(Y\) can be partitioned into those vertices which lie equidistant from \(g_+\) and \(g_-\), and those which lie closer to one gluing vertex than the other. The first subset consists precisely of those \(v\) satisfying (\ref{eq:equidist}); the second---thanks to the presence of the edge joining \(g_+\) and \(g_-\)---projects to the gluing set.

Our main theorem concerns twists that possess a relaxed version of this hybrid property: every vertex of the subgraph \(H\) either projects to the gluing set or else satisfies (\ref{eq:equidist}).

\begin{defn}\label{def:syc_twist}
A \textbf{sycamore twist} is a generalized Whitney twist \((G,H, K, \alpha)\) satisfying two additional conditions:
\begin{itemize}
\item \(K\) is convex in \(X\) and \(Y\).
\item Every vertex \(h \in H\) which does \emph{not} project to \(K\) satisfies \(d_H(h,k) =  d_H(h, \alpha(k))\) for every \(k \in K\).
\end{itemize}
\end{defn}

Not every pair of graphs that differ by sycamore twist can be related by a Whitney twist, as the following example shows.

\begin{example}\label{eg:syc_v_whitney}
Consider the graphs \(X\) and \(Y\) depicted in \Cref{fig:good_v_whitney}. They differ by a sycamore twist \((G,H,K,\alpha)\): blue vertices and edges belong to \(G\); violet vertices and edges to \(H\); and green vertices and edges to \(K = G \cap H\). The map \(\alpha\) interchanges \(v_5\) and \(v_6\), fixing \(v_4\) and \(v_7\). The vertices \(v_8\) and \(v_{11}\) project to \(K\), while \(v_9\) and \(v_{10}\) satisfy equation (\ref{eq:equidist}) with respect to each vertex in \(K\).

The graphs \(X\) and \(Y\) cannot be related by a Whitney twist. If they could, then the two gluing vertices would form a vertex cut in \(X\) (and in \(Y\)): deleting those two vertices and their incident edges would disconnect the graph. The graph \(X\) contains exactly 12 two-element vertex cuts. (One can count them by hand, or with a few lines of code.) Thus, there are in principle 12 graphs to which it can be related by a Whitney twist. However, each of those graphs is in fact isometric to \(X\); in particular, none of them is isometric to \(Y\).
\end{example}

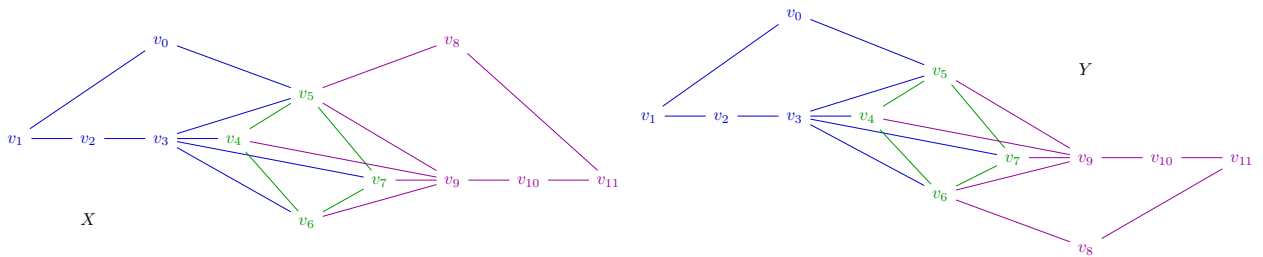
\begin{figure}[!h]
\adjustbox{scale=0.65,center}{
\begin{tikzcd}
& & \color{blue}{v_0} \arrow[blue, no head]{drr} \arrow[blue, no head]{ddll} & & & & 
\color{violet}{v_8} \arrow[violet, no head]{dll} \arrow[violet, no head]{ddddrr} \\
&  & & &[-15pt] 
\color{green}{v_5} \arrow[green, no head]{dl} \arrow[green, no head]{dddr}\\[-5pt]
\color{blue}{v_1} \arrow[blue, no head]{r} 
 & \color{blue}{v_2} \arrow[blue, no head]{r} & 
\color{blue}{v_3} \arrow[blue, no head]{r}  \arrow[blue, no head]{urr} \arrow[blue, no head]{dddrr} \arrow[blue, no head]{ddrrr} & 
\color{green}{v_4} & [-15pt]  & &  \\[-15pt]
 \\[-10pt]
& & & & &
\color{green}{v_7} & 
\color{violet}{v_9} \arrow[violet, no head]{uulll} \arrow[violet, no head]{uuull} \arrow[violet, no head]{l} \arrow[violet, no head]{dll} & 
\color{violet}{v_{10}} \arrow[violet, no head]{l} 
& \color{violet}{v_{11}} \arrow[violet, no head]{l}
\\[-10pt]
& X & & & 
\color{green}{v_6} \arrow[green, no head]{uuul} \arrow[green, no head]{ur}
\end{tikzcd}
\begin{tikzcd}
&& 
\color{blue}{v_0} \arrow[blue, no head]{drr} \arrow[blue, no head]{ddll} & & & & \\
&& & &[-15pt] 
\color{green}{v_5} \arrow[green, no head]{dl} \arrow[green, no head]{dddr} & & Y \\[-5pt]
\color{blue}{v_1} \arrow[blue, no head]{r} 
&\color{blue}{v_2} \arrow[blue, no head]{r} & 
\color{blue}{v_3} \arrow[blue, no head]{r} \arrow[blue, no head]{urr} \arrow[blue, no head]{dddrr} \arrow[blue, no head]{ddrrr} &
\color{green}{v_4} &[-15pt] & & \\[-15pt]
 \\[-10pt]
&& & & & 
\color{green}{v_7} & 
\color{violet}{v_9} \arrow[violet, no head]{uulll} \arrow[violet, no head]{uuull} \arrow[violet, no head]{l} \arrow[violet, no head]{dll} &
\color{violet}{v_{10}} \arrow[violet, no head]{l}
& \color{violet}{v_{11}} \arrow[violet, no head]{l}
 \\[-10pt]
&& & & 
\color{green}{v_6} \arrow[green, no head]{uuul} \arrow[green, no head]{ur} \\
&& & & & & 
\color{violet}{v_8} \arrow[violet, no head]{uurr} \arrow[violet, no head]{ull}
\end{tikzcd}
}
\caption{A sycamore twist which is not a Whitney twist. The sycamore twist is named after the winged seed pod of the sycamore tree, which twists in the air as it falls. (British schoolchildren call these `helicopters'.) The idea is that the vertices which project to \(K\) resemble wings projecting from a central core.}
\label{fig:good_v_whitney}
\end{figure}

Equally, not every Whitney twist is a sycamore twist. An example is given by any Whitney twist in which the gluing vertices are non-adjacent yet are connected by an edge-path in \(X\) (or equivalently in \(Y\))---for in this case, the subgraph on the gluing vertices is not convex. However, in a Whitney twist along \emph{adjacent} vertices, this subgraph is guaranteed to be convex. Indeed, by the discussion preceding \Cref{def:syc_twist}, any such twist is a sycamore twist with respect to the subgraph \(K\) comprising the two gluing vertices and the edge between them; the map \(\alpha\) flips the vertices of \(K\).

This paper's main theorem (\Cref{thm:gt_mag}) says that magnitude is invariant under sycamore twists. \Cref{eg:syc_v_whitney} shows that this is a proper generalization of Leinster's result. We will prove it using a homological argument, but the proof will {not} imply that magnitude homology is invariant under sycamore twists. Instead, the homological algebra serves to simplify the counting of paths in \(X\) and \(Y\). After establishing, in \Cref{sec:props_good_twist}, a few basic facts about the metrics on \(X\) and \(Y\), we partition the set of paths in each graph into those which are `twistable'---meaning that there is a particular bijection between the vertices of \(X\) and \(Y\) which preserves their length---and those which are not. (This is the subject of \Cref{sec:twist_exc}.) By showing that non-twistable paths generate a chain complex which is contractible, we discover (in Sections \ref{sec:mag_twists} and \ref{sec:excise}) that we can discount them when it comes to calculating magnitude.


\section{Properties of sycamore twists}\label{sec:props_good_twist}

This section establishes those properties of a sycamore twist which will facilitate our analysis of its magnitude. We begin by fixing notation and recording basic information about the distance functions on the twisted graphs \(X\) and \(Y\). We will see that the set of non-gluing vertices in \(H\) can be partitioned into two subsets, each with a convenient property derived from the defining properties of a sycamore twist. Finally, we relate \(X\) and \(Y\) by two functions on their vertex sets which restrict locally to isometries; in \Cref{sec:twist_exc} these functions will be used to establish a relationship between the magnitude complexes of \(X\) and \(Y\).

\begin{defn}[Notation for gluing vertices]\label{def:embed_v}
Let \((G, H, K, \alpha)\) be a sycamore twist. Given a vertex \(v \in K\), we denote its image \(\iota_G(v)\) in \(G\) simply by \(v\), and do the same for its image \(\iota_H(v)\) in \(H\). We also use \(v\) to denote the vertex \(\iota_G(v) = \iota_H(v)\) in \(X\) and the vertex \(\iota_G(v) = \iota_H(\alpha(v))\) in \(Y\).
\end{defn}

Using this notation, for all \(g \in G\), \(v \in K\) and \(h\in H\) we have
\[d_X(g,v) = d_G(g,v) \text{ and } d_X(h,v) = d_H(h,v)\]
while for \(g \in G\),  \(v \in K\) and \(h \in H \backslash K\) we have
\[d_Y(g,v) = d_G(g,v) \text{ and } d_Y(h,v) = d_H(h, \alpha(v)).\]
It will be useful to record a description of the other distances in \(X\) and in \(Y\).

\begin{lem}\label{lem:dist_in_twist}
Given vertices \(u\) and \(v\) in \(X\) we have
\[d_X(u,v) = \begin{cases}
d_G(u,v) & u,v \in G \\
d_H(u,v) & u,v \in H \backslash K \\
\min_{k \in K} (d_G(u,k) + d_H(k,v)) & u \in G \text{ and } v \in H \backslash K.
\end{cases}\]
Given vertices \(u\) and \(v\) in \(Y\) we have
\[d_Y(u,v) = \begin{cases}
d_G(u,v) & u,v \in G \\
d_H(u,v) & u,v \in H \backslash K \\
 \min_{k \in K} (d_G(u,k) + d_H(\alpha(k),v)) & u \in G \text{ and } v \in H \backslash K.
\end{cases}\]
\end{lem}

\begin{proof}
The statements concerning \(u,v \in G\) and \(u,v \in H \backslash K\) follow from Lemma 4.3 in \cite{LeinsterMagnitude2019}, which says that if \(K = G \cap H\) is convex in \(X\) (or in \(Y\)) then both \(G\) and \(H\) are also convex in \(X\) (respectively in \(Y\)). Take \(u \in G\) and \(v \in H \backslash K\); we want to see that 
\begin{equation}\label{dist_twist_1}
d_X(u,v) = \min_{k \in K} (d_G(u,k) + d_H(k,v))
\end{equation}
and
\begin{equation}\label{dist_twist_2}
d_Y(u,v) =  \min_{k \in K} (d_G(u,k) + d_H(\alpha(k),v)).
\end{equation}

Let \(k\) be any vertex in \(K\). Suppose there exists an edge-path in \(G\) between \(u\) and \(k\), and an edge-path in \(H\) between \(k\) and \(v\): the concatenation of any such pair of paths gives an edge path from \(u\) to \(v\) in \(X\), so all three distances are finite and we have \(d_X(u,v) \leq d_G(u,k) + d_H(k,v)\). On the other hand, if there is \emph{no} edge-path in \(G\) between \(u\) and \(k\) then \(d_G(u,k) = \infty\) and the inequality still holds; the same is true if there is no edge-path in \(H\) between \(k\) and \(v\). Thus
\[d_X(u,v) \leq \min_{k \in K}(d_G(u,k) + d_H(k,v)).\]

For the reverse inequality, we may assume \(d_X(u,v)\) is finite. Let \(\vec{x} = (x_0, \ldots, x_n)\) be a minimal edge path from \(u\) to \(v\) in \(X\). Any path between a vertex in \(G\) and one in \(H\) must pass through at least one vertex in \(K\) (this is Lemma 4.4 in \cite{LeinsterMagnitude2019}). Let \(x_i\) be such a vertex in \(\vec{x}\). Then \((x_0, \ldots, x_i)\) is a minimal edge path between \(u\) and \(x_i\) in \(X\), and as \(G\) is convex in \(X\) this implies \(d_G(u,x_i) = d_X(u,x_i) = i\). Similarly, \(d_H(x_i, v) = d_X(x_i, v) = n-i\). So \(d_X(u,v) = d_G(u,x_i) + d_H(x_i,v)\) for some vertex \(x_i \in K\), and hence
\[d_X(u,v) \geq \min_{k \in K}(d_G(u,k) + d_H(k,v)).\]
This proves that (\ref{dist_twist_1}) holds; the same argument establishes (\ref{dist_twist_2}), after making use of the fact that \(d_Y(k,v) = d_H(\alpha(k),v)\) for all \(k \in K\) and \(v \in H \backslash K\).
\end{proof}

\begin{defn}
Let \((G,H, K, \alpha)\) be a sycamore twist. The vertices of \(K\) are the \textbf{gluing} vertices. Vertices in \(H \backslash K\) which project to \(K\) will be called \textbf{biased} vertices; we will denote the set of biased vertices by \(H_*\). A vertex \(v\) in \(H \backslash K\) which does \emph{not} project to \(K\) will be called \textbf{\(\alpha\)-neutral} or just \textbf{neutral}; we will denote the set of neutral vertices by \(H_0\).
\end{defn}

The terminology is motivated by \Cref{lem:neutral_equal}, which says that neutral vertices do not notice the difference between the metrics in \(X\) and in \(Y\). \Cref{lem:proj_in_twist} tells us that biased vertices have a complementary property: in both \(X\) and \(Y\) the vertices in \(H_*\) project to \(G\).

\begin{lem}\label{lem:neutral_equal}
Let \((G,H, K, \alpha)\) be a sycamore twist and let \(v\) be a neutral vertex in \(H\). For every \(u \in G \cup H\) we have
\[d_X(u,v) = d_Y(u,v).\]
\end{lem}

\begin{proof}
By the definition of a sycamore twist, each neutral vertex \(v\) satisfies
\[d_H(k,v) = d_H(\alpha(k),v)\]
for all \(k \in K\). The result follows upon comparing the expressions for \(d_X(u,v)\) and \(d_Y(u,v)\) in \Cref{lem:dist_in_twist}.
\end{proof}

\begin{defn}\label{def:between}
A vertex \(w\) in a graph \(G\) is \textbf{between} vertices \(u\) and \(v\) if
\[d_G(u,v) = d_G(u,w) + d_G(w,v).\]
\end{defn}

\begin{lem}\label{lem:proj_in_twist}
Let \((G,H, K, \alpha)\) be a sycamore twist. Then \(h \in H\) projects to \(K\) in \(H\) if and only if \(h\) projects to \(G\) in \(X\) and in \(Y\).
\end{lem}

\begin{proof}
The ``if" statement is clear; we must prove the converse. Suppose \(h \in H\) projects to \(K\) in \(H\), and let \(g\) be a vertex in \(G\). If \(g\) is a gluing vertex then
\[d_X(h,g) = d_H(h,g) = d_H(h, \pi(h)) + d_H(\pi(h), g) = d_X(h, \pi(h)) + d_X(\pi(h), g),\]
so we need to see that in \(X\) the vertex \(\pi(h)\) lies between \(h\) and every non-gluing vertex of \(G\). Take \(g \in G \backslash K\). Via \Cref{lem:dist_in_twist} we have
\begin{align*}
d_X(g,h) &= \min_{k\in K} (d_G(g,k) + d_H(k,h)) \\
&=\min_{k\in K} (d_G(g,k) + d_H(k,\pi(h)) + d_H(\pi(h), h)) \\
&= \min_{k\in K} (d_G(g,k ) + d_H(k,\pi(h))) + d_H(\pi(h), h) \\
&= d_X(g, \pi(h)) + d_X(\pi(h), h)
\end{align*}
as required. Thus, \(h\) projects to \(G\) in \(X\).

The argument for \(Y\) is essentially the same. If \(g \in G\) is a gluing vertex, then
\begin{align*}
d_Y(g,h) = d_H(\alpha(g),h) &= d_H(\alpha(g), \pi(h)) + d_H(\pi(h), h) \\
&= d_Y(g, \alpha^{-1}(\pi(h))) + d_Y(\alpha^{-1}(\pi(h)), h)
\end{align*}
so we need to see that \(\alpha^{-1}(\pi(h))\) lies between \(h\) and every non-gluing vertex of \(G\). Take \(g \in G \backslash K\). We have
\begin{align*}
d_Y(g,h) &= \min_{k\in K} (d_G(g,k) + d_H(\alpha(k),h)) \\
&=\min_{k\in K} (d_G(g,k) + d_H(\alpha(k), \pi(h)) + d_H(\pi(h), h)) \\
&=\left(\min_{k\in K} (d_G(g,k) + d_H(\alpha(k), \pi(h)))\right) + d_Y(\alpha^{-1}(\pi(h)), h) \\
&=\left(\min_{k\in K} (d_Y(g,k) + d_Y(k, \alpha^{-1}(\pi(h))))\right) + d_Y(\alpha^{-1}(\pi(h)), h) \\
&=d_Y(g, \alpha^{-1}(\pi(h))) + d_Y(\alpha^{-1}(\pi(h)), h)
\end{align*}
where the final line holds by the triangle inequality in \(Y\), since \(\alpha^{-1}(\pi(h))\) belongs to \(K\). Thus,  if \(h\) projects to \(K\) in \(H\), it projects to \(G\) in \(X\) and in \(Y\).
\end{proof}

Whenever \(X\) and \(Y\) differ by a twist \((G,H, K, \alpha)\) we can define bijective functions \(\tau_G, \tau_H\) from the vertex set of \(X\) to that of \(Y\) by 
\[\tau_G(v) = v \text{ for all } v \in X\]
and
\[\tau_H(v) = \begin{cases} v & \text{if \(v\) is not a gluing vertex} \\
\alpha^{-1}(v) & \text{if \(v\) is a gluing vertex}. \end{cases}\]
(Here, as above, we are using the notation for gluing vertices established in \Cref{def:embed_v}.)

The function \(\tau_G\) can be thought of as fixing the subgraph \(G\) and twisting \(H\), while \(\tau_H\) fixes \(H\) and twists \(G\). Except in degenerate cases, neither \(\tau_G\) nor \(\tau_H\) defines a graph homomorphism on the whole of \(X\). Rather, they are constructed so that, provided the twist is a \emph{sycamore} twist, each map is a graph homomorphism---in fact, a bijective isometry---when restricted to one part of the cover \(\{G \cup H_0, H\}\).

\begin{lem}\label{lem:tau_isom}
Let \((G,H, K, \alpha)\) be a sycamore twist. Then the function \(\tau_G\) restricts to an isometry on \(G \cup H_0\) and \(\tau_H\) restricts to an isometry on \(H\). 
\end{lem}

\begin{proof}
It is clear that \(\tau_G\) restricts to an isometry on \(G\) and on \(H_0\) separately, and \Cref{lem:neutral_equal} tells us that for all \(g \in G\) and \(h \in H_0\) we have 
\[d_Y(\tau_G(g), \tau_G(h)) = d_X(g,h),\]
so \(\tau_G\) restricts to an isometry on \(G \cup H_0\).

Similarly, it is clear that \(\tau_H\) restricts to an isometry on \(H \backslash K\). It also restricts to an isometry on \(K\): given any \(k, k' \in K\), we have
\begin{align*}
d_Y(\tau_H(k), \tau_H(k')) &= d_Y(\alpha^{-1}(k),\alpha^{-1}(k')) \\
&= d_H(\alpha(\alpha^{-1}(k)), \alpha(\alpha^{-1}(k'))) \\
&= d_H(k,k') \\
&= d_X(k,k')
\end{align*}
as required. Finally, if \(h \in H \backslash K\) and \(k \in K\) then
\[d_Y(\tau_H(h), \tau_H(k)) = d_Y(h, \alpha^{-1}(k)) = d_H(h, \alpha(\alpha^{-1}(k))) = d_H(h,k) = d_X(h,k).\]
Thus, \(\tau_H\) restricts to an isometry on \(H\).
\end{proof}

\Cref{sec:twist_exc} will make use of one more lemma which, although we state it in the notation of our present setup, is a general and elementary fact about projection.

\begin{lem}\label{lem:pi_between}
Suppose \(u,v \in G\) and \(w \in H_*\). Then \(v\) lies between \(u\) and \(w\) if and only if it lies between \(u\) and \(\pi(w)\).
\end{lem}

\begin{proof}
Since \(\pi(w)\) is between \(u\) and \(w\), we have
\[d(u, \pi(w)) = d(u,w) - d(\pi(w), w).\]
If, and only if, \(v\) is between \(u\) and \(w\), we can rewrite the right hand side as \(d(u,v) + d(v,w) - d(\pi(w), w)\). Since \(\pi(w)\) is between \(v\) and \(w\), this yields
\[d(u, \pi(w)) = d(u,v) + d(v, \pi(w)) + d(\pi(w), w) - d(\pi(w), w) = d(u,v) + d(v, \pi(w)). \qedhere\]
\end{proof}

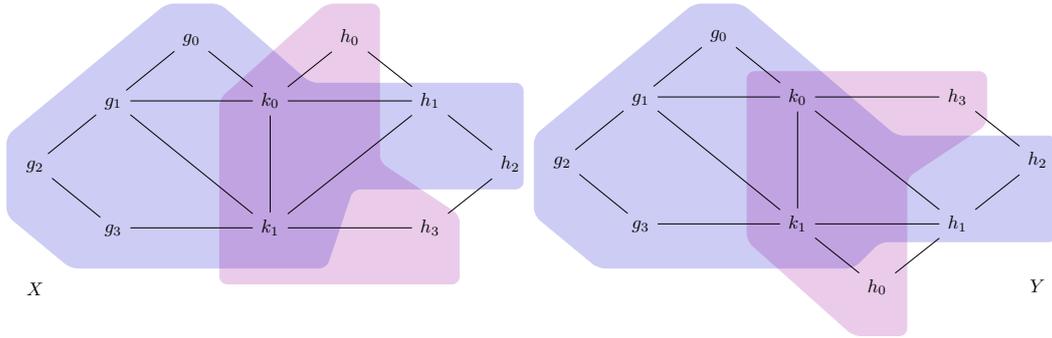
\begin{figure}[h]
\adjustbox{scale=0.7, center}{
\begin{tikzcd}[execute at begin picture={
{{\draw [draw=blue, fill=blue, draw opacity=.2, fill opacity=.2, rounded corners] (0.7,1.5) -- (1.5,1.5) -- (4.7,1.5) -- (4.7, -0.5) -- (1.5, -0.5) -- (1,-2) -- (-3.8,-2) -- (-5, -1) -- (-5,0.5) -- (-2, 3) -- (-1, 3) -- cycle;}}
{{\draw [draw=violet, fill=violet, draw opacity=.2, fill opacity=.2, rounded corners] (-1, 1.2) -- (1, 3) -- (2,3) -- (2,0) -- (3.5, -1) -- (3.5, -2.3) -- (-1, -2.3) -- (-1, -1) -- cycle;}}
  }]
& & \color{black}{g_0} \arrow[black, no head]{dl} \arrow[black, no head]{dr} & & \color{black}{h_0} \arrow[black, no head]{dl} \arrow[black, no head]{dr}\\
& \color{black}{g_1} \arrow[black, no head]{rr} \arrow[black, no head]{ddrr} & & \color{black}{k_0}  \arrow[black, no head]{dd} \arrow[black, no head]{rr} & & \color{black}{h_1} \\
\color{black}{g_2} \arrow[black, no head]{ur}  \arrow[black, no head]{dr} &&&&&& \color{black}{h_2} \arrow[black, no head]{ul}  \arrow[black, no head]{dl} \\
& \color{black}{g_3} \arrow[black, no head]{rr} & & \color{black}{k_1}  \arrow[black, no head]{rr}  \arrow[black, no head]{uurr} & & \color{black}{h_3} \\
X
\end{tikzcd}
\begin{tikzcd}[execute at begin picture={
{{\draw [draw=blue, fill=blue, draw opacity=.2, fill opacity=.2, rounded corners] (1.8,0.5) -- (4.8,0.5) -- (4.8, -1.5) -- (1.5, -1.5) -- (1,-2) -- (-3.8,-2) -- (-5, -1) -- (-5,0.5) -- (-2, 3) -- (-1.1, 3) -- cycle;}}
{{ \begin{scope}[yscale=-1,xscale=1,yshift=8]
	\draw [draw=violet, fill=violet, draw opacity=.2, fill opacity=.2, rounded corners]
	(-1, 1.2) -- (1, 3) -- (2,3) -- (2,0) -- (3.5, -1) -- (3.5, -2) -- (-1, -2) -- (-1, -1) -- cycle;
	\end{scope}}}
  }]
& & \color{black}{g_0} \arrow[black, no head]{dl} \arrow[black, no head]{dr}\\
& \color{black}{g_1} \arrow[black, no head]{rr} \arrow[black, no head]{ddrr} & & \color{black}{k_0}  \arrow[black, no head]{dd} \arrow[black, no head]{rr}  \arrow[black, no head]{ddrr} & & \color{black}{h_3} \\
\color{black}{g_2} \arrow[black, no head]{ur}  \arrow[black, no head]{dr} &&&&&& \color{black}{h_2} \arrow[black, no head]{ul}  \arrow[black, no head]{dl} \\
& \color{black}{g_3} \arrow[black, no head]{rr} & & \color{black}{k_1}  \arrow[black, no head]{rr} & & \color{black}{h_1} \\
 & & & & \color{black}{h_0} \arrow[black, no head]{ul} \arrow[black, no head]{ur} && Y
\end{tikzcd}
}
\caption{The subgraphs \(G'\) (shaded in blue) and \(H'\) (shaded in violet) do not cover \(X\) or \(Y\) as graphs: the edges \((h_0, h_1)\) and \((h_2,h_3)\) are missing from \(G' \cup H'\).}\label{fig:G'H'}
\end{figure}

\begin{rmk}\label{rmk:not_excision}
Given a sycamore twist \((G,H,K, \alpha)\), let \(G'\) denote the subgraph of \(X\) whose vertices are those in \(G \cup H_0\) and let \(H'\) denote the subgraph whose vertices are those in \(H_* \cup K\). Then the vertex sets of \(G'\) and \(H'\) cover the vertex set of \(X\), and every vertex in \(H'\) projects to \(G' \cap H' = K\).

If \((X; G',H')\) were a projecting decomposition, and the same were true for the similar triple \((Y;G',H')\), then our main theorem would follow from the inclusion-exclusion formula. But this is not the case. In general \(G'\) and \(H'\) do not cover \(X\) as graphs: there may exist edges \((u,v)\) in \(X\) such that \(u\) belongs to \(G' \backslash (G' \cap H')\) and \(v\) belongs to \(H' \backslash (G' \cap H')\), in which case \((u,v)\) does not appear in \(G' \cup H'\). \Cref{fig:G'H'} provides an example.
\end{rmk}


\section{Twistable and non-twistable paths}\label{sec:twist_exc}

Consider the graphs \(X\) and \(Y\) in \Cref{fig:diff_paths}. There are evident differences in the sorts of paths that exist in each graph: for instance, \(Y\) has two 1-paths of length 3---the paths \((g_0, h_1)\) and \((h_1, g_0)\)---while \(X\) has none. On the other hand, the two graphs also have many paths in common. For example, any path which never visits the biased vertex \(h_1\) evidently has the same length whether it is regarded as belonging to \(X\) or to \(Y\). And the same is true for any path which never visits \(h_1\) without visiting the neutral vertex \(h_0\) immediately beforehand and immediately afterwards---this follows from the fact, characteristic of neutral vertices, that \(d_X(h_0,-) = d_Y(h_0,-)\) (\Cref{lem:neutral_equal}).

\begin{figure}[h]
\adjustbox{scale=.7, center}{
\begin{tikzcd}
& \color{blue}{g_0} \arrow[blue, no head]{dr} & & \color{violet}{h_1} \arrow[violet, no head]{dl} && \color{blue}{g_0} \arrow[blue, no head]{dr} & & &\\
X & & \color{green}{k_0} \arrow[blue, no head]{dl} \arrow[violet, no head]{dr} \arrow[green, no head]{dd} & & & & \color{green}{k_0} \arrow[blue, no head]{dl} \arrow[violet, no head]{dr} \arrow[green, no head]{dd} & & Y\\
& \color{blue}{g_1} \arrow[blue, no head]{dr}&& \color{violet}{h_0} \arrow[violet, no head]{dl}& & \color{blue}{g_1} \arrow[blue, no head]{dr}& & \color{violet}{h_0} \arrow[violet, no head]{dl}\\
& & \color{green}{k_1} &&&& \color{green}{k_1} \arrow[violet, no head]{dr}\\
& &&&&&& \color{violet}{h_1}
\end{tikzcd}
}
\caption{The graps \(X\) and \(Y\) differ by a sycamore twist along the graph \(k_0\)---\(k_1\).}
\label{fig:diff_paths}
\end{figure}
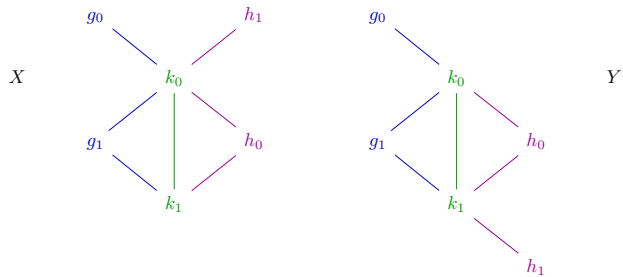

\Cref{prop:nonstick_biject}, below, extrapolates from examples of this type. We use the vertex functions \(\tau_G\) and \(\tau_H\) defined in \Cref{sec:props_good_twist} to construct a length-preserving bijection between those paths in \(X\) and \(Y\) which never cross from \(G\) to \(H_*\) (or \emph{vice versa}) without passing through a neutral vertex. Paths with this property will be called `twistable'; the next few definitions formalize this notion.

\begin{defn}
Let \(X\) be any graph and \(U\) a subset of its vertices. A path \(\vec{x} = (x_0, \ldots, x_k)\)
in \(X\) will be said to \textbf{visit \(n\) vertices in \(U\)} if there are \(n\) indices, \(i_1, \ldots, i_n\), such that the vertex \(x_{i_j}\) belongs to \(U\). We will say \(\vec{x}\) is \textbf{contained in \(U\)} if \emph{every} vertex \(x_0, \ldots, x_k\) belongs to \(U\). A \textbf{subpath} of \(\vec{x}\) is a string \((x_i, x_{i+1}, \ldots, x_j)\) of consecutive vertices in \(\vec{x}\), where \(0 \leq i \leq j \leq k\).
\end{defn}

\begin{defn}
Let \((G,H,K, \alpha)\) be a sycamore twist, and let \(\vec{x} = (x_0, \ldots, x_k)\) be a path in \(X\) (or in \(Y\)). We will say \(\vec{x}\) is \textbf{flat} if it is contained in \(G \cup H_0\) or contained in \(H\).
\end{defn}

Every flat path can be mapped in a length-preserving manner from \(X\) to \(Y\) using one or other of the vertex functions \(\tau_G\) and \(\tau_H\).

\begin{lem}\label{lem:flat_biject}
Suppose \(X\) and \(Y\) differ by a sycamore twist. For each \(k \geq 0\) there is a length-preserving bijection between flat \(k\)-paths in \(X\) and flat \(k\)-paths in \(Y\).
\end{lem}

\begin{proof}
Given a flat path \(\vec{x}\) in \(X\), let \(T(\vec{x})\) be the path in \(Y\) defined by 
\[T(\vec{x}) = \begin{cases}
\tau_G(\vec{x}) & \text{if } \vec{x} \text{ is contained in } G \cup H_0 \\
\tau_H(\vec{x}) & \text{otherwise.}
\end{cases}\]
Clearly, if \(\vec{x}\) is a \(k\)-path then \(T(\vec{x})\) is a \(k\)-path too. Since \(\tau_G\) restricts to an isometry on \(G \cup H_0\) and \(\tau_H\) restricts to an isometry on \(H\) (by \Cref{lem:tau_isom}), the path \(T(\vec{x})\) is flat and its length is the same as that of \(\vec{x}\).

Thus, we have a length-preserving function
\[T: \{\text{flat \(k\)-paths in \(X\)}\} \to \{\text{flat \(k\)-paths in \(Y\)}\}.\]
We can construct a function \(T'\) in the other direction in the same manner, replacing \(\tau_G\) by \(\tau_G^{-1}\) and \(\tau_H\) by \(\tau_H^{-1}\). The maps \(T\) and \(T'\) are mutually inverse, proving the lemma.
\end{proof}

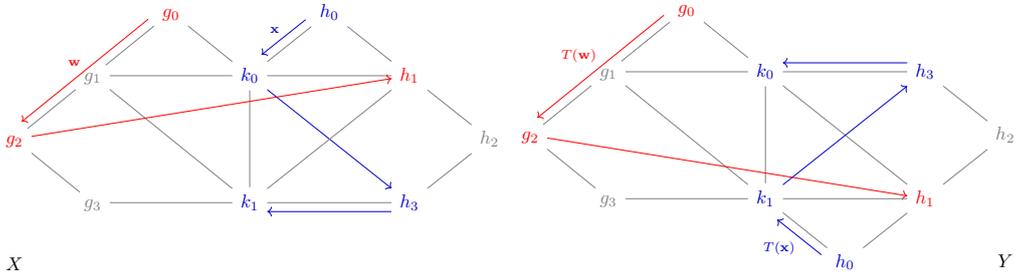
\begin{figure}[h]
\adjustbox{scale=0.7, center}{
\begin{tikzcd}
& & \color{red}{g_0} \arrow[red, shift right=2, swap]{ddll}{\vec{w}} \arrow[gray, no head]{dl} \arrow[gray, no head]{dr} & & \color{blue}{h_0} \arrow[blue, shift right=2, swap]{dl}{\vec{x}} \arrow[gray, no head]{dl} \arrow[gray, no head]{dr}\\
& \color{gray}{g_1} \arrow[gray, no head]{rr} \arrow[gray, no head]{ddrr} & & \color{blue}{k_0} \arrow[blue]{ddrr} \arrow[gray, no head]{dd} \arrow[gray, no head]{rr} & & \color{red}{h_1} \\
\color{red}{g_2} \arrow[red]{urrrrr} \arrow[gray, no head]{ur}  \arrow[gray, no head]{dr} &&&&&& \color{gray}{h_2} \arrow[gray, no head]{ul}  \arrow[gray, no head]{dl} \\
& \color{gray}{g_3} \arrow[gray, no head]{rr} & & \color{blue}{k_1}  \arrow[gray, no head]{rr}  \arrow[gray, no head]{uurr} & & \color{blue}{h_3} \arrow[blue, shift left=2]{ll} \\
X
\end{tikzcd}
\begin{tikzcd}
& & \color{red}{g_0} \arrow[red, shift right=2, swap]{ddll}{T(\vec{w})} \arrow[gray, no head]{dl} \arrow[gray, no head]{dr}\\
& \color{gray}{g_1} \arrow[gray, no head]{rr} \arrow[gray, no head]{ddrr} & & \color{blue}{k_0}  \arrow[gray, no head]{dd} \arrow[gray, no head]{rr}  \arrow[gray, no head]{ddrr} & & \color{blue}{h_3} \arrow[blue, shift right=2]{ll} \\
\color{red}{g_2} \arrow[red]{drrrrr}  \arrow[gray, no head]{ur}  \arrow[gray, no head]{dr} &&&&&& \color{gray}{h_2} \arrow[gray, no head]{ul}  \arrow[gray, no head]{dl} \\
& \color{gray}{g_3} \arrow[gray, no head]{rr} & & \color{blue}{k_1} \arrow[blue]{uurr} \arrow[gray, no head]{rr} & & \color{red}{h_1} \\
 & & & & \color{blue}{h_0} \arrow[blue, shift left=2]{ul}{T(\vec{x})} \arrow[gray, no head]{ul} \arrow[gray, no head]{ur} && Y
\end{tikzcd}
}
\caption{On the left we see two flat paths in the graph \(X\) from \Cref{fig:whit_1}. The path \(\vec{w}\) (in red) is contained in \(G \cup H_0\) and the path \(\vec{x}\) (in blue) is contained in \(H\). On the right we see their images \(T(\vec{w})\) and \(T(\vec{x})\) in \(Y\).}\label{fig:flat_path}
\end{figure}

Suppose we are given a flat path \(\vec{x} = (x_0, \ldots, x_k)\) which is contained in \(G \cup H_0\) and a flat path \(\vec{x}' = (x_0', \ldots, x_n')\) which is contained in \(H\), with the property that \(x_k = x_0'\) and this vertex is neutral. Then we can concatenate the paths \(\vec{x}\) and \(\vec{x}'\) in \(X\), and since \(\tau_G(x_k) = \tau_H(x_0')\) (as \(\tau_G\) and \(\tau_H\) agree on all but the gluing vertices) we can also concatenate the paths \(T(\vec{x})\) and \(T(\vec{x}')\) in \(Y\). Length sums over concatenation of paths, so we have that
\[L(T(\vec{x}) * T(\vec{x}') ) = L(T(\vec{x})) + L(T(\vec{x}')) = L(\vec{x}) + L(\vec{x}') = L(\vec{x} * \vec{x}').\]
By the same reasoning, \(T\) can be extended in a length-preserving manner to any sequence of finitely many flat paths concatenated at neutral vertices. We now give a name to such concatenations.

\begin{defn} 
Let \((G,H,K, \alpha)\) be a sycamore twist. A path \(\vec{x}\) in X (or in \(Y\)) is \textbf{twistable} if it can be decomposed as a {concatenation} of {paths}
\[\vec{x} = \vec{x}_0 * \vec{x}_2 * \cdots * \vec{x}_m\]
such that each path \(\vec{x}_i\) is flat and each point of concatenation is a neutral vertex. We will call such a decomposition a \textbf{maximal decomposition into flat subpaths} if \emph{every} neutral vertex visited by \(\vec{x}\) is a point of concatenation.
\end{defn}

In order to extend the function \(T\) in a {canonical} manner from flat to twistable paths, we need to specify a canonical decomposition of each twistable path into flat subpaths. The following lemma says that this is possible.

\begin{lem}\label{lem:max_decomp}
Every twistable path has a unique maximal decomposition into flat subpaths.
\end{lem}

\begin{proof}
Let \(\vec{x}\) be a twistable path in \(X\). If \(\vec{x}\) visits no neutral vertices then \(\vec{x}\) itself must be flat, and in this case there is nothing to show.  On the other hand, if \(\vec{x}\) visits at least one neutral vertex, we can decompose it as follows.

Extract the ordered list of neutral vertices visited by \(\vec{x}\)---say,
\[
(x_{i_0}, \ldots, x_{i_k}).
\]
Let \(\vec{x}_0\) be the subpath \((x_0,\ldots, x_{i_0})\) of \(\vec{x}\), let \(\vec{x}_k\) be the subpath \((x_{i_k}, \ldots, x_n)\), and for each \(0 < j < k\) let \(\vec{x}_j\) be the subpath \((x_{i_j}, \ldots, x_{i_{j+1}})\). Then
\begin{equation}\label{eq:max}
\vec{x} = \vec{x}_0 * \cdots * \vec{x}_k
\end{equation}
and in this decomposition each point of concatenation is neutral, while every neutral vertex visited by \(\vec{x}\) occurs as a point of concatenation. We just need to see that every subpath \(\vec{x}_j\) is flat.

As \(\vec{x}\) is twistable we can find \emph{some} decomposition of \(\vec{x}\) into {flat} subpaths:
\begin{equation}\label{eq:flat}
\vec{x} = \vec{x}'_0 * \cdots * \vec{x}'_m
\end{equation}
Since each point of concatenation in (\ref{eq:flat}) is neutral, each of the paths \(\vec{x}_0,\ldots, \vec{x}_k\) in (\ref{eq:max}) must be a subpath of \(\vec{x}'_i\) for some \(0 \leq i \leq m\). And since each subpath \(\vec{x}'_i\) in (\ref{eq:flat}) is flat---that is, contained either in \(G \cup H_0\) or in \(H\)---the same must be true for each \(\vec{x}_j\). Thus, (\ref{eq:max}) is a maximal decomposition into flat subpaths, and this decomposition is evidently unique as such.
\end{proof}

\Cref{lem:max_decomp} lets us extend the bijection \(T\) established in \Cref{lem:flat_biject} to a bijection between the twistable paths in \(X\) and those in \(Y\).

\begin{prop}\label{prop:nonstick_biject}
Suppose \(X\) and \(Y\) differ by a sycamore twist. Then for each \(k \geq 0\) there is a length-preserving bijection between twistable \(k\)-paths in \(X\) and twistable \(k\)-paths in \(Y\).
\end{prop}

\begin{proof}
Given a twistable \(k\)-path \(\vec{x}\) in \(X\), let
\[\vec{x} = \vec{x}_0 * \cdots * \vec{x}_m\]
be its maximal decomposition into flat subpaths. Let \(T(\vec{x})\) be the \(k\)-path in \(Y\) defined by 
\begin{equation}\label{eq:nonstick_1}
T(\vec{x}) = T(\vec{x}_0) * \cdots * T(\vec{x}_m)
\end{equation}
where \(T(\vec{x}_i)\) is specified, for each flat subpath \(\vec{x}_i\), as in the proof of \Cref{lem:flat_biject}.

Since \(\tau_G\) preserves the vertex set \(G \cup H_0\) and \(\tau_H\) preserves the vertex set \(H\), and both maps fix the neutral vertices, (\ref{eq:nonstick_1}) is a decomposition of \(T(\vec{x})\) into flat subpaths and thus \(T(\vec{x})\) is twistable. Indeed, since all its neutral vertices are points of concatenation, (\ref{eq:nonstick_1}) is the maximal decomposition of \(T(\vec{x})\) into flat subpaths.

Moreover, we have
\[L(T(\vec{x})) = \sum_{i=0}^m L(T(\vec{x}_i)) = \sum_{i=0}^m L(\vec{x}_i) = L(\vec{x}),\]
which says that the map 
\[T: \{\text{twistable \(k\)-paths in \(X\)}\} \to \{\text{twistable \(k\)-paths in \(Y\)}\}\]
is length-preserving. We can construct map \(T'\) in the other direction by extending, in the same fashion, the function \(T'\) of \Cref{lem:flat_biject}. The maps \(T\) and \(T'\) are mutually inverse, proving the proposition.
\end{proof}

Having established a bijection between the twistable paths in \(X\) and those in \(Y\), our aim in Sections \ref{sec:mag_twists} and \ref{sec:excise} is to prove that every path which is \emph{not} twistable can be disregarded for the purposes of calculating magnitude. For this, we will need a more positive characterization of the `non-twistable' paths. The remainder of this section provides that characterization.

\begin{defn}
A path \(\vec{x}=(x_0, \ldots, x_k)\) in \(X\) or \(Y\) is said to be \textbf{sticky} if \(x_0\) is a vertex in \(G \backslash K\) and \(x_k\) is a vertex in \(H_*\), or \textit{vice versa}, and \(x_{1}, \ldots, x_{k-1}\) are all gluing vertices. 
\end{defn}

Since a sticky path always visits both \(G \backslash K\) and \(H_*\), no flat path can have a {subpath} which is sticky. A partial converse also holds: 

\begin{lem}\label{lem:flat_v_creased}
Let \(\vec{x}\) be a path which is not flat and which visits no neutral vertices except perhaps its endpoints. Then \(\vec{x}\) has a sticky subpath.
\end{lem}

\begin{proof}
As it is not flat, the path \(\vec{x} = (x_0, \ldots x_k)\) must visit at least one vertex in \(G \backslash K\) and at least one vertex in \(H_*\). Identify vertices \(x_{i} \in G \backslash K\) and \(x_j \in H_*\) in \(\vec{x}\) such that \(|j-i|\) is minimal. We may assume without loss of generality that \(i < j\). As \(|j-i|\) is minimal, the vertices \(x_{i+1},\ldots,x_{j-1}\) must all belong to \(K \cup H_0\); in fact, as \(\vec{x}\) visits no neutral vertices (except perhaps its endpoints), they must all belong to \(K\). Thus, the subpath \((x_i, \ldots x_j)\) is sticky.
\end{proof}

Using the lemma, we can characterize the `non-twistable' paths in \(X\) and \(Y\) as those which possess at least one sticky subpath.

\begin{prop}\label{prop:exc_v_twist}
A path is twistable if and only if has no sticky subpath.
\end{prop}

\begin{proof}
First, suppose \(\vec{x}\) is a twistable path, and let \(\vec{x} = \vec{x}_0 * \cdots * \vec{x}_k\) be a decomposition into flat subpaths. If \(\vec{x}\) were to have a sticky subpath, say \(\vec{w}\), then---as a sticky path visits no neutral vertices---\(\vec{w}\) must be a subpath of \(\vec{x}_i\) for some \(0 \leq i \leq k\). But no flat path can have a sticky subpath; thus, \(\vec{x}\) has no sticky subpath.

Conversely, suppose \(\vec{x}\) has no sticky subpath. If \(\vec{x}\) visits no neutral vertices, then \Cref{lem:flat_v_creased} implies that \(\vec{x}\) must be flat. On the other hand, if \(\vec{x}\) visits at least one neutral vertex, we can decompose it by concatenating at every neutral vertex as in \Cref{lem:max_decomp} (equation (\ref{eq:max})). To see that every subpath \(\vec{x}_j\) in this decomposition is flat, observe that for each \(j\) we either have \(\vec{x}_j = (x_{j_0}, x_{j_1})\) where \(x_{j_0}\) and \(x_{j_1}\) are both neutral, or else \(\vec{x}_j = (x_{j_0}, \ldots, x_{j_m})\) where \(m > 1\) and none of the vertices \(x_{j_1}, \ldots, x_{j_{m-1}}\) is neutral. In the first case \(\vec{x}_j\) is contained in \(H_0\), so is flat. In the second case, \(\vec{x}_j\) is flat by \Cref{lem:flat_v_creased}.

Thus we have a decomposition into flat subpaths, so \(\vec{x}\) is twistable.
\end{proof}


\section{Magnitude is invariant under sycamore twists}\label{sec:mag_twists}

Having established that the twistable paths of a given length in \(X\) correspond bijectively with those in \(Y\), we would like to be able to disregard the paths that are not twistable; by \Cref{prop:exc_v_twist}, the non-twistable paths are those which possess at least one sticky subpath.

Ideally, one might hope to show that the non-twistable paths span an acyclic subcomplex of the magnitude complex and thus do not contribute to its homology---or, therefore, to its Euler characteristic. The trouble is that the non-twistable paths \emph{do not} span a subcomplex of \(MC(X)\) or \(MC(Y)\), acyclic or otherwise. A sticky subpath can easily be destroyed, for instance if removing one of its endpoints brings its gluing vertices into contact with a neutral vertex; \Cref{fig:not_subcx} gives an example. Consequently, the subset of generators which contain a sticky subpath is not closed under the boundary operator.

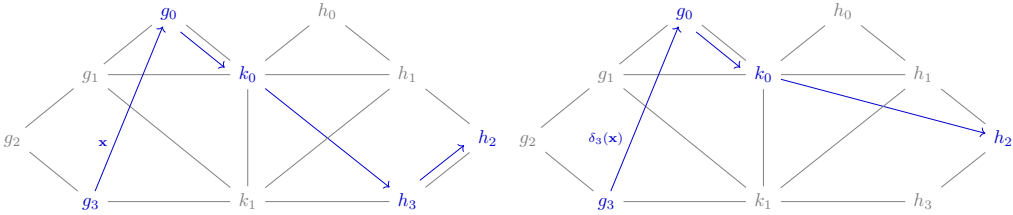
\begin{figure}[h]
\adjustbox{scale=0.7, center}{
\begin{tikzcd}
& & \color{blue}{g_0} \arrow[gray, no head]{dl} \arrow[blue, shift right=2]{dr} \arrow[gray, no head]{dr} & & \color{gray}{h_0} \arrow[gray, no head]{dl} \arrow[gray, no head]{dr}\\
& \color{gray}{g_1} \arrow[gray, no head]{rr} \arrow[gray, no head]{ddrr} & & \color{blue}{k_0}  \arrow[gray, no head]{dd} \arrow[gray, no head]{rr} \arrow[blue]{ddrr} & & \color{gray}{h_1} \\
\color{gray}{g_2} \arrow[gray, no head]{ur}  \arrow[gray, no head]{dr} &&&&&& \color{blue}{h_2} \arrow[gray, no head]{ul}  \arrow[gray, no head]{dl} \\
& \color{blue}{g_3} \arrow[gray, no head]{rr}  \arrow[blue, near start]{uuur}{\vec{x}} & & \color{gray}{k_1}  \arrow[gray, no head]{rr}  \arrow[gray, no head]{uurr} & & \color{blue}{h_3} \arrow[blue, shift left=2]{ur}
\end{tikzcd}
\begin{tikzcd}
& & \color{blue}{g_0} \arrow[gray, no head]{dl} \arrow[blue, shift right=2]{dr} \arrow[gray, no head]{dr} & & \color{gray}{h_0} \arrow[gray, no head]{dl} \arrow[gray, no head]{dr}\\
& \color{gray}{g_1} \arrow[gray, no head]{rr} \arrow[gray, no head]{ddrr} & & \color{blue}{k_0}  \arrow[gray, no head]{dd} \arrow[gray, no head]{rr} \arrow[blue]{drrr} & & \color{gray}{h_1} \\
\color{gray}{g_2} \arrow[gray, no head]{ur}  \arrow[gray, no head]{dr} &&&&&& \color{blue}{h_2} \arrow[gray, no head]{ul}  \arrow[gray, no head]{dl} \\
& \color{blue}{g_3} \arrow[gray, no head]{rr}  \arrow[blue, near start]{uuur}{\delta_3(\vec{x})} & & \color{gray}{k_1}  \arrow[gray, no head]{rr}  \arrow[gray, no head]{uurr} & & \color{gray}{h_3}
\end{tikzcd}
}
\caption{Here, in blue, we see two paths in the graph \(X\) from \Cref{fig:whit_1}. In the path \(\vec{x}\) on the left there is a sticky subpath, \((g_0, k_0, h_3)\). On the right is the path \(\delta_3(\vec{x})\), which is flat: it is contained in \(G \cup H_0\).}\label{fig:not_subcx}
\end{figure}

To circumvent this problem we are going to filter the magnitude complex in such a way that, when we consider any term-by-term subquotient of the filtration, the non-twistable paths \emph{do} span a subcomplex. That subcomplex will turn out to be acyclic. The proof of its acyclicity, which is inspired by Hepworth and Willerton's proof of the excision theorem for magnitude chains (\cite{HepworthWillerton2017}, Theorem 28), is given in \Cref{sec:excise}.

The idea behind the filtration is as follows. Since a sticky subpath may be destroyed by removing one of its endpoints, and since the endpoints of a sticky subpath always belong to the vertex set \(X \backslash (H_0 \cup K)\), we would like to engineer things so that removing a vertex in \(X \backslash (H_0 \cup K)\) from a generator always has the effect of sending that generator to zero. To achieve this, we filter the magnitude chains by the number of vertices they visit in \(X \backslash (H_0 \cup K)\).

Formally, for each \(m \in \bb{N}\), let \(T_m^X\) denote the subcomplex of \(MC(X)\) spanned in degree \(k \in \bb{N}\) by \(k\)-paths that visit at most \(m\) vertices in \(X \backslash (H_0 \cup K)\). We can filter \(MC(X)\) by these subcomplexes---
\begin{equation}\label{eq:T_filt}
0 \subseteq T_0^X \subseteq T_1^X \subseteq \cdots T_m^X \subseteq \cdots \subseteq MC(X)
\end{equation}
---and we can filter \(MC(Y)\) in a similar manner. We have \(T_{m,k}^X = MC_k(X)\) for all \(k < m\), and the same for \(Y\). What's more, in any given length grading \(\ell\) the magnitude chain complex vanishes above homological degree \(k=\ell\) (since any \(k\)-path in a graph has length at least \(k\)), so in every length grading \(\ell\) this filtration stabilizes after the \(\ell^\th\) term.

For each \(m\), write \(Q_m^X = T_m^X / T_{m-1}^X\). The complex \(Q_m^X\) is freely generated in degree \(k\) by paths \(\vec{x} = (x_0, \ldots, x_k)\) that visit \textit{exactly} \(m\) vertices in \(X \backslash (H_0 \cup K)\). The boundary map is \(d = \sum_{i=1}^{k-1} (-1)^i \delta_i\) where \(\delta_i\) removes the vertex \(x_i\) unless \(x_i\) belongs to \(X \backslash (H_0 \cup K)\) or removing it reduces the length of the path; in either of those cases, \(\delta_i(\vec{x}) = 0\).

\begin{lem}\label{lem:ex_sub}
For each \(m \in \bb{N}\), the set of non-twistable paths in \(X\) which visit exactly \(m\) vertices in \(X \backslash (H_0 \cup K)\) spans a subcomplex of \(Q_m^X\).
\end{lem}

\begin{proof}
Let \(\vec{x} = (x_0, \ldots, x_k)\) be a generator of \(Q_m^X\) in degree \(k\). By \Cref{prop:exc_v_twist}, \(\vec{x}\) has at least one sticky subpath; let \((x_i, \ldots, x_j)\) be a sticky subpath.

Removing any of the vertices \(x_p\) for \(p < i\) or \(p > j\) will return a path still containing the sticky subpath \((x_i, \ldots, x_j)\), provided it does not send the generator to zero. Meanwhile, removing any of the vertices \(x_{i+1}, \ldots, x_{j-1}\) will return a sticky subpath with one fewer vertex, unless it sends the generator to zero. And since the endpoints \(x_i\) and \(x_j\) belong to \(X \backslash (H_0 \cup K)\), removing either one of them will certainly send the generator to zero. Thus, for every \(p \in \{1, \ldots, k-1\}\), either \(\delta_p(\vec{x})\) has a sticky subpath or \(\delta_p(\vec{x}) = 0\).
\end{proof}

Exactly same argument proves the same lemma for \(Y\):

\begin{lem}\label{lem:ex_sub_Y}
For each \(m \in \bb{N}\), the set of non-twistable paths in \(Y\) which visit exactly \(m\) vertices in \(Y \backslash (H_0 \cup K)\) spans a subcomplex of \(Q_m^Y\). \qed
\end{lem}

At heart of the main theorem is the following proposition, which says that the subcomplexes spanned by the non-twistable paths are acyclic. The proof is given in \Cref{sec:excise}.

\begin{prop}\label{prop:ex_acyc}
Let \(E_m^X\) denote the subcomplex of \(Q_m^X\) spanned by non-twistable paths, and define \(E_m^Y\) similarly for \(Y\). For every \(m \in \bb{N}\), the complexes \(E_m^X\) and \(E_m^Y\) are both acyclic. \qed
\end{prop}

\begin{cor}\label{cor:same_ec}
In each length grading \(\ell\) we have \(\chi(Q_m^{X, \ell}) = \chi(Q_m^{Y,\ell})\).
\end{cor}

\begin{proof}
\Cref{prop:ex_acyc} tells us that for each \(m\) and \(\ell\) the subcomplex \(E_m^{X,\ell}\) of \(Q_m^{X,\ell}\) is acyclic; consequently, its Euler characteristic vanishes. This lets us write
\[\chi(Q_m^{X,\ell}) = \chi(Q_m^{X,\ell}) - \chi(E_m^{X,\ell})\]
and thus, counting generators,
\begin{align}
\chi(Q_m^{X,\ell}) = & \sum_{k=0}^\infty (-1)^k \#
\left\{ \begin{array}{c}
\text{\(k\)-paths of length \(\ell\) in \(X\) which visit} \\
\text{exactly \(m\) vertices in \(X \backslash (H_0 \cup K)\)}
\end{array}\right\} \nonumber
\\ & - \sum_{k=0}^\infty (-1)^k \#
\left\{ \begin{array}{c}
\text{non-twistable \(k\)-paths of length \(\ell\) in \(X\) which} \\
\text{visit exactly \(m\) vertices in \(X \backslash (H_0 \cup K)\)} 
\end{array}\right\} \nonumber \\
= & \sum_{k=0}^\infty (-1)^k \#
\left\{ \begin{array}{c}
\text{twistable \(k\)-paths of length \(\ell\) in \(X\) which} \\
\text{visit exactly \(m\) vertices in \(X \backslash (H_0 \cup K)\)}
\end{array}\right\}.\label{cor:same_ec_1}
\end{align}
By the same reasoning,
\begin{align}
\chi(Q_m^{Y,\ell}) = & \sum_{k=0}^\infty (-1)^k \#
\left\{ \begin{array}{c}
\text{twistable \(k\)-paths of length \(\ell\) in \(Y\) which} \\
\text{visit exactly \(m\) vertices in \(Y \backslash (H_0 \cup K)\)}.
\end{array}\right\}.\label{cor:same_ec_2}
\end{align}
Since \(\tau_G\) and \(\tau_H\) each fix every vertex outside \(K\) and map \(K\) to itself, both functions preserve the vertex set \(H_0 \cup K\) and its complement. Consequently, the bijection 
\[T: \{\text{twistable \(k\)-paths in \(X\)}\} \to \{\text{twistable \(k\)-paths in \(Y\)}\}\]
constructed in \Cref{prop:nonstick_biject} preserves, as well as the length of a path, the number of vertices it visits outside \(H_0 \cup K\). Using \(T\) to compare each summand in (\ref{cor:same_ec_1}) and (\ref{cor:same_ec_2}), we can conclude that \(\chi(Q_m^{X, \ell}) = \chi(Q_m^{Y,\ell})\).
\end{proof}

We are now equipped to prove the main theorem.

\begin{thm}\label{thm:gt_mag}
Let \(X\) and \(Y\) be graphs which differ by a sycamore twist. Then
\[\Mag(X) = \Mag(Y).\]
\end{thm}

\begin{proof}
By \Cref{thm:MH_mag} we have
\begin{equation}\label{magX}
\Mag(X) = \sum_{\ell \in \bb{N}} \chi(MC_\bullet^\ell(X)) q^\ell
\end{equation}
and 
\begin{equation}\label{magY}
\Mag(Y) = \sum_{\ell \in \bb{N}} \chi(MC_\bullet^\ell(Y)) q^\ell.
\end{equation}
As the filtration
\[0 \subseteq T_0^X \subseteq T_1^X \subseteq \cdots T_m^X \subseteq \cdots \subseteq MC(X)\]
and the similar filtration of \(MC(Y)\) both stabilize, in any given length grading \(\ell\), after the \(\ell^\th\) term, the additivity of Euler characteristic implies that we can calculate the \(\ell^\th\) coefficient in (\ref{magX}) as
\[\chi(MC_\bullet^\ell(X)) = \chi(T_\ell^{X, \ell}) = \sum_{k=0}^{\ell} \chi(Q_k^{X, \ell})\]
the \(\ell^\th\) coefficient in (\ref{magY}) as
\[\chi(MC_\bullet^\ell(Y)) = \chi(T_\ell^{Y, \ell}) = \sum_{k=0}^{\ell} \chi(Q_k^{Y, \ell}).\]
Now \Cref{cor:same_ec} {tells us} that for every \(k, \ell \in \bb{N}\) we have \(\chi(Q_k^{X, \ell}) = \chi(Q_k^{Y, \ell})\) and hence \(\chi(MC_\bullet^\ell(X)) = \chi(MC_\bullet^\ell(Y))\) for every \(\ell\). The theorem follows on comparing coefficients in (\ref{magX}) and (\ref{magY}).
\end{proof}


\section{Proof of Proposition 6.3}\label{sec:excise}

It remains to prove that for each \(m \in \bb{N}\) the subcomplex of \(Q_m^X\) spanned by the non-twistable paths is acyclic. The main idea of this proof is adapted from Hepworth and Willerton's Lemma 9.2 in \cite{HepworthWillerton2017}, a central component of their excision theorem: it exploits the projection property of vertices in \(H_*\) to construct a contracting chain homotopy.

Essentially, given a non-twistable generator \(\vec{x}\) whose first sticky subpath \((x_i, \ldots, x_j)\) \textbf{crosses from \(H_*\) to \(G \backslash K\)}---meaning that \(x_i\) belongs to \(H_*\) and \(x_j\) belongs to \(G \backslash K\)---the homotopy \(s\) inserts the vertex \(\pi(x_i)\) immediately after \(x_i\). By definition of stickiness, the vertex \(x_{i+1}\) must belong to \(G\), so \(\pi(x_j)\) lies between \(x_{i}\) and \(x_{i+1}\) and thus we have
\begin{equation}\label{eq:contract1}
\delta_{i+1}(s(\vec{x})) = \vec{x}.
\end{equation}
If, instead, the first sticky subpath in \(\vec{x}\) crosses from \(G\backslash K\) to \(H_*\), then \(s\) inserts \(\pi(x_j\)) immediately before \(x_j\), and we have 
\begin{equation}\label{eq:contract2}
\delta_j(s(\vec{x})) = \vec{x}.
\end{equation}
Ultimately the fact that \(s\) is a contraction comes down to equations \ref{eq:contract1} and \ref{eq:contract2}.

\begin{propex_acyc}
Let \(E_m^X\) denote the subcomplex of \(Q_m^X\) spanned by non-twistable paths, and define \(E_m^Y\) similarly. For every \(m \in \bb{N}\), the complexes \(E_m^X\) and \(E_m^Y\) are both acyclic.
\end{propex_acyc}

We will prove the statement for \(E_m^X\). The same proof goes through for \(E_m^Y\), since it depends only on the common properties of \(X\) and \(Y\) established in \Cref{sec:props_good_twist}---in particular, the fact that vertices in \(H_*\) project, in both \(X\) and \(Y\), to \(G\).

\begin{proof}
First, {observe that} \(E_m^X \cong E_m^X(G,H) \oplus E_m^X(H,G)\) where \(E_m^X(G,H)\) is spanned by  paths whose first sticky subpath crosses from \(G \backslash K\) to \(H_*\), and  \(E_m^X(H,G)\) is spanned by paths whose first sticky subpath crosses from \(H_*\) to \(G \backslash K\). We are going to establish that for every \(m \in \bb{N}\) the complexes \(E_m^X(G,H)\) and \(E_m^X(H,G)\) are both acyclic. We prove it in full for \(E_m^X(H,G)\) and sketch the proof for \(E_m^X(G,H)\). 

Given a generator \(\vec{x}\) of \(E_m^X(H,G)\), denote the index of the first point of its first sticky subpath by \(F_\vec{x}\). For each \(i \geq 0\), let \(N_\bullet(i)\) be the subcomplex of \(E_m^X(H,G)\) spanned by paths \(\vec{x}\) such that \(F_\vec{x} \leq i\). (That is, such that the subpath \((x_{0}, \ldots, x_i)\) overlaps with the first sticky subpath.) For every \(i\) we have \(N_k(i) = E_{m,k}^X(H,G)\) for all \(k \leq i\), and for each \(i\) there is an inclusion \(N_\bullet(i) \subseteq N_\bullet(i+1)\). Indeed, we can filter \(E_m^X(H,G)\) as
\[0 = N_\bullet(-1) \subseteq N_\bullet(0) \subseteq N_\bullet(1) \subseteq \cdots \subseteq N_\bullet(i) \subseteq \cdots \subseteq E_m^X(H,G)\]
and again this filtration stabilizes in every length grading. We will show that for each \(i \geq 0\) the quotient \(N_\bullet(i) / N_\bullet(i-1)\) is contractible. Thus, on passing to homology, each of these inclusions becomes an isomorphism and we can conclude that \(E_m^X(H,G)\) is acyclic.

The complex \(N_\bullet(i) / N_\bullet(i-1)\) is freely generated in degree \(k\) by paths \(\vec{x} = (x_0, \ldots, x_k)\) such that \(F_\vec{x} = i\); that is, paths whose first sticky subpath begins with the \(i^\th\) vertex. To understand the boundary operator, notice that removing any one of the vertices \(x_0, \ldots, x_{i-1}\) from \(\vec{x}\) yields a path whose first sticky subpath begins with the \((i-1)^\th\) vertex; this goes to zero in the quotient. Since \(x_i\) is the endpoint of a sticky subpath, it belongs to \(X \backslash (K \cup H_0)\), so removing it also sends the generator to zero. Thus, \(d = \sum_{j=i+1}^{k-1} (-1)^j \delta_j\) where \(\delta_j\) removes \(x_j\) unless \(x_j\) belongs to \(X \backslash (K \cup H_0)\) or removing it reduces the length of the path; in either of those cases \(\delta_j(\vec{x}) = 0\).

Define a map \(s: N_k(i)/N_k(i-1) \to N_{k+1}(i) / N_{k+1}(i-1)\) on generators \((x_0, \ldots, x_k)\) as follows. As \(x_{i}\) belongs to \(H_*\), it projects to \(G\) in \(X\). Let \(\pi(x_{i})\) denote the gluing vertex it projects through, and put
\[s(x_0, \ldots, x_k) = \begin{cases}
(-1)^{i+1} (x_0, \ldots, x_{i}, \pi(x_{i}), x_{i+1}, \ldots, x_k) & \text{if } x_{i+1} \neq \pi(x_{i}) \\
0 & \text{if } x_{i+1} = \pi(x_{i}).
\end{cases}\]
I claim this is a contracting homotopy: that \(s \circ d + d \circ s\) is the identity on \(N_k(i)/N_k(i-1)\), or equivalently that for every generator \(\vec{x} = (x_0, \ldots, x_k)\) we have
\begin{equation}\label{eq:ch_1}
\sum_{j=i+1}^{k-1} (-1)^j s \delta_j^k(\vec{x}) + \sum_{j=i+1}^{k} (-1)^j \delta_j^{k+1} s (\vec{x}) = \vec{x}.
\end{equation}

For \(j=i+2, \ldots, k-1\) we have \(s \delta_{j}^k(\vec{x}) = \delta_{j+1}^{k+1} s(\vec{x})\), so most terms on the left of (\ref{eq:ch_1}) cancel, leaving
\begin{equation}\label{eq:ch_2}
(-1)^{i+1} s \delta_{i+1}^k(\vec{x}) + (-1)^{i+1} \delta_{i+1}^{k+1} s (\vec{x}) + (-1)^{i+2} \delta_{i+2}^{k+1} s (\vec{x}).
\end{equation}

If \(x_{i+1} = \pi(x_{i})\), the second and third terms of (\ref{eq:ch_2}) vanish, leaving
\begin{align*}
(-1)^{i+1} s \delta_{i+1}^k(x_0, \ldots, x_k) &= (-1)^{i+1} s (x_0, \ldots, x_{i}, x_{i+2}, \ldots x_k) \\
&= (-1)^{2(i+1)} (x_0, \ldots, x_{i}, \pi(x_{i}), x_{i+2}, \ldots x_k) \\
&= (x_0, \ldots, x_{i}, x_{i+1}, x_{i+2}, \ldots x_k).
\end{align*}
Here, the first equation holds since \(x_{i+2}\) is in \(G\)---it's either a gluing vertex, or it's the other end of the sticky subpath---so \(x_{i+1} = \pi(x_{i})\) lies between it and \(x_{i}\).

Suppose \(x_{i+1} \neq \pi(x_{i})\); then \(x_{i+1}\) may be a gluing vertex, or it may be the other end of the sticky subpath. In the latter case it belongs to \(G \backslash K\), so \(\delta_{i+1}^k(\vec{x}) = 0 = \delta_{i+2}^{k+1}s(\vec{x})\) and (\ref{eq:ch_2}) reduces to
\begin{align*}
(-1)^{i+1}\delta_{i+1}^{k+1} s(x_0, \ldots, x_k) &= (-1)^{2(i+1)} \delta_{i+1}^{k+1} (x_0, \ldots, x_{i}, \pi(x_i), x_{i+1}, \ldots, x_k) \\
&= (x_0, \ldots, x_{i}, x_{i+1}, \ldots, x_k)
\end{align*}
where the second equation holds because \(x_{i+1} \in G\).

Finally, suppose \(x_{i+1}\) is a gluing vertex (but not the vertex \(\pi(x_{i})\)); then \(x_{i+2}\) is either a gluing vertex or an element of \(G \backslash K\). There are two cases to consider: either \(x_{i+1} \in K\) lies between \(x_{i+2} \in G\) and \(x_{i} \in H_*\), or it doesn't. In the former case, {by \Cref{lem:pi_between}} and \Cref{lem:proj_in_twist}, \(x_{i+1}\) must also lie between \(x_{i+2}\) and \(\pi(x_{i})\); in the latter case, it must not.
Thus, in the former case we have
\begin{align*}
(-1)^{i+1} s \delta_{i+1}^k(\vec{x}) & + (-1)^{i+1} \delta_{i+1}^{k+1} s (\vec{x}) + (-1)^{i+2} \delta_{i+2}^{k+1} s(\vec{x}) \\
= & (-1)^{i+1} s (x_0, \ldots, x_i, x_{i+2}, \ldots, x_k) \\
& + (-1)^{2(i+1)} \delta_{i+1}^{k+1} (x_0, \ldots, x_i, \pi(x_i), x_{i+1}, \ldots, x_k) \\
& + (-1)^{2i+3} \delta_{i+2}^{k+1} (x_0, \ldots, x_i, \pi(x_i), x_{i+1}, \ldots, x_k) \\
= & (-1)^{2i + 2} (x_0, \ldots, x_i, \pi(x_i), x_{i+2}, \ldots, x_k) \\
& + (-1)^{2(i+1)} (x_0, \ldots, x_i, x_{i+1}, \ldots, x_k) \\
& + (-1)^{2i+3} (x_0, \ldots, x_i, \pi(x_i), x_{i+2}, \ldots, x_k) \\
= & (x_0, \ldots, x_k)
\end{align*}
as required, while in the latter case we have \(\delta_{i+1}^k(\vec{x}) = 0 = \delta_{i+2}^{k+1} s(\vec{x})\), and (\ref{eq:ch_2}) reduces again to
\begin{align*}
(-1)^{i+1} \delta_{i+1}^{k+1} s(x_0, \ldots, x_k) &= (-1)^{2(i+1)} \delta_{i+1}^{k+1} (x_0, \ldots, x_{i}, \pi(x_{i}), x_{i+1}, \ldots, x_k) \\
&= (x_0, \ldots, x_k).
\end{align*}
In every case equation (\ref{eq:ch_1}) holds, and we can conclude that the complex \(E_m^X(H,G)\) is acyclic.

The proof for \(E_m^X(G,H)\) is very similar; we will sketch the set-up to highlight where things differ. This time, given a generator \(\vec{x}\), denote the index of the \emph{last} point of its {first} sticky subpath by \(L_\vec{x}\), and for each \(i \geq 0\) let \(M_\bullet(i)\) be the subcomplex of \(E_m^X(G,H)\) spanned by paths \(\vec{x}\) such that \(L_\vec{x} \geq k-i\). (That is, such that the subpath \((x_{k-i}, \ldots, x_k)\) overlaps with the first sticky subpath.) Thus we obtain a filtration
\begin{equation}\label{eq:filt_M}
0 = M_\bullet(-1) \subseteq M_\bullet(0) \subseteq M_\bullet(1) \subseteq \cdots \subseteq M_\bullet(i) \subseteq \cdots \subseteq E_m^X(G,H),
\end{equation}
which stabilizes in every length grading as before.

Now, for each \(i\), define a map \(s: M_k(i)/M_k(i-1) \to M_{k+1}(i) / M_{k+1}(i-1)\) on generators \((x_0, \ldots, x_k)\) as follows. As \(x_{k-i}\) belongs to \(H_*\), it projects to \(G\) in \(X\). Let \(\pi(x_{k-i})\) denote the gluing vertex it projects through, and put
\[s(x_0, \ldots, x_k) = \begin{cases}
(-1)^{k-i} (x_0, \ldots, x_{k-i-1}, \pi(x_{k-i}), x_{k-i}, \ldots, x_k), & x_{k-i-1} \neq \pi(x_{k-i}) \\
0, & x_{k-i-1} = \pi(x_{k-i}).
\end{cases}\]
An argument mirroring the one above proves that \(s\) is a contracting chain homotopy. Thus, on passing to homology, each inclusion in (\ref{eq:filt_M}) becomes an isomorphism, and we can conclude that \(E_m^X(G,H)\) is acyclic.
\end{proof}


\nocite{}
\printbibliography

\end{document}